\documentclass[10pt,a4paper,twoside,reqno]{amsart}
\title[]
{Twisted and singular gravitating vortices}

\author[C.-J. Yao]{Chengjian Yao}
\address{ Institute of Mathematical Sciences, ShanghaiTech University, 393 Middle Huaxia Road, Pudong, 
	Shanghai, 201210 China.}
\email{yaochj@shanghaitech.edu.cn}

\usepackage{hyperref,url, bm}
\usepackage{amssymb,latexsym}
\usepackage{amsmath,amsthm}
\usepackage[latin1]{inputenc}
\usepackage{enumerate}
\usepackage{color}
\usepackage{venndiagram}
\usepackage[all]{xy} \CompileMatrices
\SelectTips{cm}{12} 

\usepackage{verbatim} 

\usepackage{wrapfig}
\usepackage{graphicx}
\usepackage{color}

\def\YYint#1#2#3{{\setbox0=\hbox{$#1{#2#3}{\int}$}
		\vcenter{\hbox{$#2#3$}}\kern-.52\wd0}}

\theoremstyle{plain}
\newtheorem{theorem}{Theorem}[section]
\newtheorem{lemma}[theorem]{Lemma}

\newtheorem{proposition}[theorem]{Proposition}

\newtheorem*{theorem*}{Theorem}
\theoremstyle{definition}
\newtheorem{definition}[theorem]{Definition}
\newtheorem{definition-theorem}[theorem]{Definition-Theorem}

\theoremstyle{remark}
\newtheorem{remark}[theorem]{Remark}

\numberwithin{equation}{section} 

\newcommand{\tr}{\operatorname{tr}}

\newcommand{\Hom}{\operatorname{Hom}}

\newcommand{\surj}{\to\kern-1.8ex\to}

\renewcommand{\Im}{\operatorname{Im}}

\DeclareMathOperator{\osc}{Osc}

\begin{document}
	
	\begin{abstract}
	We introduce the notion of twisted gravitating vortex on a compact Riemann surface. If the genus of the Riemann surface is greater than $1$ and the twisting forms have suitable signs, we prove an existence and uniqueness result for suitable range of the coupling constant generalizing the result of \cite{Al-Ga-Ga-P} in the non twisted setting. It is proved via solving a continuity path deforming the coupling constant from $0$ for which the system decouples as twisted K\"ahler-Einstein metric and twisted vortices. Moreover, specializing to a family of twisting forms smoothing delta distribution terms, we prove the existence of singular gravitating vortices whose K\"ahler metric has conical singularities and Hermitian metric has parabolic singularities. In the Bogomol'nyi phase, we establish an existence result for singular Einstein-Bogomol'nyi equations, which represents cosmic strings with singularities.
	\end{abstract}
	
	\maketitle
	
	\setlength{\parskip}{5pt}
	\setlength{\parindent}{0pt}
	

\section{Introduction}	

The Abelian Higgs model coupled with gravity/or equivalently the \emph{Einstein-Maxwell-Higgs} theory in four dimension is the simplest theory when matter interacts with spacetime according to Einstein's Field Equation and the matter field is described by a gauge theory on this spacetime simultaneously.  Let $g$ be a Lorenztian metric with signature $(-,+,+,+)$ on a four-manifold $X$. 
The Einstein equation coupled with the Abelian Higgs model \cite{Yang} are 

\[
\left\{
\begin{array}{rll}
\text{Ric } g - \frac{1}{2}R g 
& = &
-4\pi G T\\
D_A^*D_A\bm\phi
& = &
\frac{1}{2}(|\bm\phi|^2-\tau)\bm \phi\\
D_A^*F_A
& = & 
= \frac{\sqrt{-1}}{2} \left( \langle \bm \phi, D_A\bm\phi\rangle - \langle D_A\bm\phi, \bm\phi\rangle \right)
\end{array}
\right.
\]
where the \emph{energy-momentum tensor} $T$ is given by 
\[
T_{\mu\nu}
=
g^{\mu'\nu'}F_{\mu\mu'}F_{\nu\nu'} 
+ 
\frac{1}{2}\left(  
[D_\mu\bm\phi][D_\nu\bm\phi]^* + [D_\mu\bm\phi]^*[D_\nu\bm\phi]
\right)
-
g_{\mu\nu}\mathcal{L}
\]
and $\mathcal{L}$ is the Abelian Higgs action density 
\[
\mathcal{L}
=
\frac{1}{4}|F_A|^2 + \frac{1}{2} |D_A\bm\phi|^2 + \frac{1}{8}(|\bm\phi|^2-\lambda)^2. 
\]
In this model, the matter and energy encoded in the energy-momentum tensor are generated by an $U(1)$ connection together with a cross-section $\bm\phi$ of this $U(1)$-bundle breaking the symmetry.

Following the work of \cite{CG, Yang}, the system is equivalent to the following Einstein-Bogomol'nyi equations (also known as self-dual Einstein-Maxwell-Higgs equations in the literature \cite{Ch, HS0, HS}, the solutions are known as cosmic strings):

\[
\left\{
\begin{array}{rll}
K_g - 4\pi G \mathcal{E}
& = &
0\\
D_j\bm\phi \pm \sqrt{-1} \varepsilon^k_{\phantom{k}j} D_k\bm\phi
& = &
0\\
\varepsilon^{jk}F_{jk} 
\pm
(|\bm\phi|^2 - \lambda)
& = & 
0
\end{array}
\right.
\]
in the case when $X=\mathbb{R}^{1,1}\times \Sigma$ and $g$  is a product of the flat metric on $\mathbb{R}^{1,1}$ with a metric $g$ on a surface $\Sigma$ and $L$, $A$, $\bm\phi$ are pulled back from the factor $\Sigma$. Viewing $g$ as a K\"ahler metric with respect to the complex structure determined by $g$, the second equation is translated as a holomorphicity condition on $\bm\phi$ and the system fits into the following more general system, i.e. the \emph{gravitating vortex} equations:

\begin{equation}\label{eq:smoothGV}
\left\{
\begin{array}{rcl}
i F_h + \frac{1}{2}(|\bm\phi|^2_h - \tau)\omega & = &0\\
\text{Ric } \omega - 2\alpha i\partial\bar\partial |\bm\phi|_h^2 +\alpha\tau (|\bm\phi|_h^2 -\tau)\omega & =&  c \omega
\end{array}
\right.
\end{equation}
introduced in \cite{Al-Ga-Ga2} which is a dimensional reduction of the K\"ahler-Yang-Mills equation defined in \cite{Al-Ga-Ga1}. Here, the unknowns are the pair of K\"ahler metric $\omega$ on $\Sigma$ and Hermitian metric $h$ on the holomorphic line bundle $L$, where $L$ is equipped with a holomorphic section $\bm\phi$.  

The existence of smooth solution with $c=0$ on a compact surface was obtained by Yisong Yang \cite{Yang} under an assumption of the multiplicities of the zeros of $\bm\phi$, which was later interpreted as a classical GIT (geometric invariant theory) stability condition for the $PSL(2,\mathbb{C})$ action on $S^N\mathbb{P}^1$ and proved to be necessary including the case $c>0$. This stability condition was recently proved to be  sufficient for the existence of smooth solution in the case $c>0$ in \cite{FPY}.  The case of negative $c$ on compact Riemann surfaces of genus $g\geq 1$ for a suitable range of the coupling constant was studied by \cite{Al-Ga-Ga-P}.   The Euclidean Einstein-Maxwell equation with positive cosmological constant also exhibits interesting relations with K\"ahler geometry in real dimension four \cite{LeB}.

The aim of this article is to solve the following coupled system
\begin{equation}\label{eq:cpGV}
\left\{
\begin{array}{rcl}
i F_h + \frac{1}{2}(|\bm\phi|^2_h - \tau)\omega & = & -2\pi \sum_k \alpha_k [r_k] \\
\text{Ric } \omega - 2\alpha i\partial\bar\partial |\bm\phi|_h^2 +\alpha\tau (|\bm\phi|_h^2 -\tau)\omega & =& \tilde c\omega + 2\pi \sum_j (1-\beta_j)[q_j]
\end{array}
\right.
\end{equation}
about $(\omega, h)$ which simply introduces several delta distribution terms to the system \eqref{eq:smoothGV}.  Here $\omega$ is a K\"ahler metric on the Riemann surface $\Sigma$ and $h$ is a Hermitian metric on the holomorphic line bundle $L$, and both metrics are allowed to have singularities possibly at $ \{q_1, \cdots, q_M; r_1, \cdots, r_S\}$.  The coefficients in front of the distribution terms, i.e. the corresponding weight vectors, are assumed to satisfy $\bm \beta=(\beta_1, \cdots, \beta_M) \in (0,1)^M$ and $\bm \alpha=(\alpha_1, \cdots, \alpha_S)\in \left( \mathbb{R}_{\geq 0} \right)^S$. The points $r_1, \cdots, r_S$ need not be different from $q_1,\cdots, q_M$. A solution to this system is called \emph{singular gravitating vortices} (see Definition \ref{def:singular-gravitating-vortices} for precise definition).

From the mathematical side, Biswas and Baptista in \cite{BB} proposed one motivation for the study of vortices with degenerated singularities generalizing the study of smooth vortices by Noguchi, Bradlow, and Garc\'ia-Prada ~\cite{Brad,G1,Noguchi}. We recall it here. Given a fixed background K\"ahler metric $\omega$ with conical and degenerated singularities on $\Sigma$, the existence of vortex solution to the first equation of \eqref{eq:cpGV} was studied in  \cite{BB}. Hermitian metric with degeneracies is called \emph{parabolic structure} there. If $h$ is a vortex solution  on $(\Sigma, L, \omega, \bm\phi)$ with parabolic singularities $r_k \; (k=1,2,\cdots, S)$ of order $\alpha_k=m_k\in \mathbb{N}_+$, i.e. 

\begin{equation}
	i F_h + \frac{1}{2}(|\bm\phi|_h^2 - \tau)\omega = - 2\pi \sum_k m_k [r_k]
\end{equation}

We could define a singular Hermitian metric $h_k$ on $L_{[r_k]}$ such that $|\bm t_k|^2_{h_k}\equiv 1$ (i.e. we take the Hermitian metric with curvature $F_{h_k}$ satisfies $iF_{h_k}=2\pi [r_k]$ ).  Then $h'= h\otimes \prod_k h_k^{\otimes m_k}$ defines a (bounded) Hermitian metric on $L'=L\otimes \prod_k L_{[r_k]}^{\otimes m_k}$, and $\bm\phi'=\bm\phi\otimes \prod_k \bm t_k^{\otimes m_k}$ is a holomorphic section of $L'$ with $|\bm\phi'|_{h'}^2=|\bm\phi|_h^2 \prod_k |\bm t_k|_{h_k}^{2m_k}=|\bm\phi|_h^2$, and therefore 
\begin{equation}
	iF_{h'} + \frac{1}{2}(|\bm\phi'|_{h'}^2 - \tau )\omega=0
\end{equation}
i.e. we get a solution $h'$  to the vortex equation on $(\Sigma, L', \omega, \bm\phi')$. Conversely, any smooth vortices give rise to a vortex solution with parabolic singularities when the Higgs field splits off some holomorphic factors.  Now, suppose $\beta_j = \frac{1}{n_j}$, and $\alpha_k=m_k$ for positive integers $m_1, \cdots, m_S$ and $n_1, \cdots, n_M$, then Proposition \ref{prop:branched-cover} reveals an intimate link between singular gravitating vortices and smooth gravitating vortex through branched cover of Riemann surfaces.

The following is the outline of this article. In order to solve the existence of singular gravitating vortices, we introduce a natural smoothing of the system, where the distribution terms are replaced by a family of approximating $(1,1)$-forms $\eqref{eq:rcdGV}_\varepsilon$. These are called twisted vortices and twisted gravitating vortex (see Definition \ref{def:twisted-vortices} and \ref{def:twisted-gravitating-vortices}). The existence and uniqueness (Theorem \ref{thm:existence-uniqueness}) of twisted vortices is established in section \ref{sect:twisted-vortices}, generalizing the existence and uniqueness of vortices \cite{Brad,G1,Noguchi}. 
Inspired by the continuity method used in \cite{Al-Ga-Ga-P, FPY}, we study a continuity path deforming the coupling constant $\alpha$ while fixing the twisting forms in section \ref{sect:main}. The openness is established when $g(\Sigma)\geq 2$ and the twisting forms have suitable signs; the closedness is proved by generalizing the a priori estimates in \cite{Al-Ga-Ga-P}, and therefore the existence of twisted gravitating vortex is established in this setting. The uniqueness of twisted gravitating vortex follows from the uniqueness of twisted vortices established in Theorem \ref{thm:existence-uniqueness} and the known uniqueness of twisted K\"ahler-Einstein metric, for instance cf. \cite{Yao}. In the final section, we show uniform  $C^\gamma$ estimates on the Riemann surfaces and uniform higher order estimates away from the singularities for this family of continuity paths $\eqref{eq:rcdGV}_\varepsilon$. By taking subsequential limit (in the same spirit as \cite{Yao, Yao2}) when the smoothing parameter $\varepsilon\to 0$, we obtain existence of singular gravitating vortices (Theorem \ref{thm:existence-singular-gravitating-vortices}).

In analogy to the smooth case, singular gravitating vortices with $\tilde c=0$ gives rise to solution of Abelian Higgs model coupled with gravity such that the four-dimensional spacetime metric has conical singularities along copies of $\mathbb{R}^{1,1}$ and the energy-momentum tensor represents delta distribution along those surfaces. Light rays could split when passing across those world-sheets, and generate double images from the observers' point of view.  As a potential model of spacetime formed according to spontaneous symmetry breaking during the rapid cooling process after the big bang, cosmic string with conical singular spacetime attracts a lot interests \cite{CG, HS0, HS, Mee}. This particularly interesting case is studied in the last section and we obtain existence of solutions to singular Einstein-Bogomol'nyi equations under assumption ({\bf{A}}) (see Theorem \ref{thm:singular-EB}).  The uniqueness of solutions and the interesting case of positive $\tilde c$ are left to future study. \\

{\bf{Acknowledgement}}: 
We would like to thank Gac\'ia-Fern\'andez and Vamsi Pingali for their valuable suggestions.

\section{Approximating Continuity Paths \& Twisted vortices}\label{sect:twisted-vortices}

	\begin{definition}[Hermitian metric with parabolic singularities]
	Let $\Sigma$ be a Riemann surface and $\{r_1, \cdots, r_S\}$ be a set of points on $\Sigma$. Let $\bm \alpha=(\alpha_1, \cdots, \alpha_S)$ be a tuple of positive real numbers. 
	The Hermitian metric $h$ on $L$ is said to have parabolic singularity of order $2\alpha_k$ at $r_k$ if $h= |z|^{2\alpha_k} e^{2\varphi}$ for some H\"older continuous function $\varphi$ on $\Sigma$, and holomorphic coordinate $z$ centered at $r_k$. 
\end{definition}

Let  $\omega$ be a K\"ahler metric on $\Sigma$ and $\{q_1, \cdots, a_M\}$ be a set of points on $\Sigma$. Let $\beta_j\in (0,1), j=1,\cdots, M$ be real numbers.  The metric $\omega$ is said to have conical singularity with angle $2\pi\beta_j$ at $q_j$ if locally near $q_j$ we have $|\bm s_j|^{2-2\beta_j} \omega=\digamma \omega_0$ with $\digamma$ being positive and H\"older continuous and $\omega_0$ smooth. There are extensive studies on constant curvature metric with conical singularities on Riemann surfaces in the past decades, cf. \cite{Mc,Tr, HX}.

Denote $\bm t_k$ a defining section of the divisor $[r_k]$ for $k=1,\cdots, S$. Let $h_0$ be a smooth Hermitian metric on $L$, if the Hermitian metric $h$ on $L$ has parabolic singularity of order $2\alpha_k$ at $r_k$ for $k=1,2, \cdots, S$, then $h=h_0 \prod_k |\bm t_k|^{2\alpha_k} e^{2\tilde f}$ for some smooth Hermitian metric $h_0$ and H\"older continuous function $\tilde f$ globally defined on $\Sigma$, where $|\bm t_k|$ is the norm of the defining section for $L_{[r_k]}$ under some smooth Hermitian metric. Denote $ f= \tilde f+\frac{1}{2}\sum_k \alpha_k \log |\bm t_k|^2 $. In this way, the potential $f$  has logarithmic pole at $r_k$ while the reduced potential $\tilde f$ does not. For any holomorphic section $\bm\phi$ of $L$, $
|\bm\phi|_h^2 
= 
|\bm\phi|^2 e^{2f}= |\bm\phi|^2 \prod_k |\bm t_k|^{2\alpha_k} e^{2\tilde f}$. 
The number 
\[
\tilde N=N + \sum_k \alpha_k.
\]
is called the \emph{parabolic degree} in the literature \cite{BB}. 

Fix $\omega_0$ a constant curvature metric on $\Sigma$ and assume the unknown K\"ahler metric $\omega=\omega_0+2i\partial\bar\partial u \in [\omega_0]$. 
 Let $\tilde \chi(\Sigma) = \chi(\Sigma) - \sum_{j} (1-\beta_j)$ and  
\begin{equation}
\tilde c
=
\frac{2\pi}{\text{Vol}_{\omega_0}} 
\Big( \tilde \chi(\Sigma) - 2\alpha\tau \tilde N \Big)
=
\frac{2\pi}{\text{Vol}_{\omega_0}} 
\Big( 2 -2g(\Sigma)- \sum_j (1-\beta_j) - 2\alpha\tau  N - 2\alpha\tau \sum_k \alpha_k\Big).
\end{equation}

 Let $h_0$ be a constant curvature metric on $L$ and $h_k$ to be a constant curvature metric on $L_{[r_k]}$, i.e. 

\begin{equation}
\text{Ric } \omega_0
=
\frac{2\pi}{\text{Vol}_{\omega_0}}  \chi(\Sigma)\omega_0\;\;, \; 
iF_{h_0}=\frac{2\pi}{\text{Vol}_{\omega_0}}N \omega_0\;\;,\;
iF_{h_k} = \frac{2\pi}{\text{Vol}_{\omega_0}} \omega_0
\end{equation}
Let $\bm s_j$ be a defining section of the divisor $[q_j]$ and fix a Hermitian metric on this line bundle with constant curvature, denote $|\bm s_j|$ the norm under this metric. Using $\Delta$ to represent the Laplacian operator of the metric $\omega_0$. The Poincar\'e-Lelong formula $i\partial\bar\partial \log |\bm t_k|^2 = - \frac{2\pi}{\text{Vol}_{\omega_0}} \omega_0+ 2\pi [r_k]$,  $i\partial\bar\partial \log |\bm s_j|^2 = - \frac{2\pi}{\text{Vol}_{\omega_0}} \omega_0+ 2\pi [q_j]$
together with $\text{Ric }\omega = \text{Ric }\omega_0 - i\partial\bar\partial \log (1-\Delta u)$ transforms the singular gravitating vortices equation \eqref{eq:cpGV} about $(\omega, h)$ into the a PDE system for $u$ and $\tilde f$:
\begin{equation}
\label{eq:cpGV-PDE}
\begin{split}
\Delta \tilde f
+
\frac{1}{2} ( |\bm\phi|_h^2 -\tau ) (1-\Delta u)
& = - \frac{2\pi}{\text{Vol}_{\omega_0}}\tilde N\\
\Delta u +
\lambda \frac{e^{4\alpha\tau \tilde f - 2\alpha |\bm\phi|_h^2 - 2\tilde c u}}{\prod_j |\bm s_j|^{2(1-\beta_j)}} 
& = 1
\end{split}
\end{equation}
where $\lambda>0$ is a real constant. Notice that in case $\tilde c\neq 0$, the constant $\lambda$ could be absorbed into the term $e^{-2\tilde c u}$, therefore there is no loss in assuming $\lambda=1$; while in the case $\tilde c=0$, varying $\lambda$ gives rise to different solutions. These two cases will be treated differently in the following materials.

\begin{definition}[singular gravitating vortices]\label{def:singular-gravitating-vortices}
	Let $\omega$ be a K\"ahler metric with conical singularities of angle $2\pi \beta_j$ at $q_j$ and $h$ be a Hermitian metric with parabolic singularity of order $2\alpha_k$ at $r_k$. If $(\omega, h)$ satisfies Equation \eqref{eq:cpGV} or equivalent \eqref{eq:cpGV-PDE}, it is called \emph{singular gravitating vortices}. 
\end{definition}

\subsection{Smoothing of the current equations}

Let $L_{[q_j]}, L_{[r_k]}$ and $\bm s_j, \bm t_k$ be with the same meaning as before.  In order to solve \eqref{eq:cpGV} we consider a family of approximating equations. Let $\varepsilon>0$, write 
\[
\chi_j^\varepsilon
= \frac{2\pi}{\text{Vol}_{\omega_0}} \omega_0 + i\partial\bar\partial \log (|\bm s_j|^2 +\varepsilon)\;,\;\;
\chi_k^{'\varepsilon}
=\frac{2\pi}{\text{Vol}_{\omega_0}} \omega_0+i\partial\bar\partial \log ( |\bm t_k|^2+\varepsilon ),
\] 
then $\chi_j^\varepsilon$ is a smoothing of the integration current  $2\pi [q_j]$, and $\chi_k^{'\varepsilon}$ is a smoothing of $2\pi [r_k]$. Moreover, $\chi_j^\varepsilon$ and $\chi_k^{'\varepsilon}$ are positive $(1,1)$-forms \cite{CDS, Yao, Yao2}. We are going to firstly solve a pair of smooth metrics $(\omega_{\varepsilon, \alpha}, h_{\varepsilon, \alpha})$ satisfying 

\begin{equation}\label{eq:rcdGV}
\left\{
\begin{array}{rcl}
iF_{h_{\varepsilon,\alpha}} + \frac{1}{2}(|\bm\phi|^2_{h_{\varepsilon,\alpha}} - \tau)\omega_{\varepsilon,\alpha} & = & -\sum_k \alpha_k\chi_k^{'\varepsilon} \\
\text{Ric } \omega_{\varepsilon,\alpha} - 2\alpha i\partial\bar\partial |\bm\phi|_{h_{\varepsilon,\alpha}}^2 +\alpha\tau (|\bm\phi|_{h_{\varepsilon,\alpha}}^2 -\tau)\omega_{\varepsilon,\alpha} & =& \tilde c\omega_{\varepsilon,\alpha} + \sum_j (1-\beta_j)\chi_j^\varepsilon
\end{array}
\right.
\end{equation}

Here $\omega_{\varepsilon,\alpha}=\omega_0+2i\partial\bar\partial u_{\varepsilon,\alpha}$ and $h_{\varepsilon,\alpha}=h_0 e^{2f_{\varepsilon,\alpha}}
=:h_0 \prod_k(|\bm t_k|^2+\varepsilon)^{\alpha_k} e^{2\tilde f_{\varepsilon,\alpha}}$. We denote this system by $\eqref{eq:rcdGV}_{\varepsilon,\alpha}$ to emphasize the dependence on the coupling constant $\alpha$ and the smoothing parameter $\varepsilon$. 
As above, writing $\Phi_{\varepsilon,\alpha}=|\bm\phi|_{h_{\varepsilon,\alpha}}^2$, then the system is transformed to the equivalent system about the potential functions $(u_{\varepsilon,\alpha}, \tilde f_{\varepsilon,\alpha})$:

\begin{equation}\label{eq:smooth-system}
\begin{split}
\Delta \tilde f_{\varepsilon,\alpha}
+
\frac{1}{2} \left(\Phi_{\varepsilon,\alpha} -\tau \right) (1-\Delta u_{\varepsilon,\alpha})
& = -\frac{2\pi}{\text{Vol}_{\omega_0}}\tilde N\\
\Delta u_{\varepsilon,\alpha} + \frac{e^{4\alpha\tau \tilde f_{\varepsilon,\alpha} - 2\alpha \Phi_{\varepsilon,\alpha} - 2\tilde c u_{\varepsilon,\alpha}}}{\prod_j (|\bm s_j|^2+\varepsilon)^{1-\beta_j} } 
& = 1
\end{split}
\end{equation}
This system of semilinear  elliptic PDEs, denoted by $\eqref{eq:smooth-system}_{\varepsilon,\alpha}^{sm}$, will be crucially used in the a priori estimates and regularity study. From now on, we assume $\text{Vol}_{\omega_0}=2\pi$ for notation simplicity. However, we should notice that the system $\eqref{eq:cpGV}$ does not enjoy a scaling symmetry on $\omega$ when $\alpha,\tau$ are fixed, i.e. solutions inside the K\"ahler class $[\omega_0]$ does not scale to a solution inside the K\"ahler class $2[\omega_0]$ for istance. 

\subsection{Twisted vortices}
When the coupling constant $\alpha=0$, the system $\eqref{eq:rcdGV}_{\varepsilon, \alpha}$ decouples into

\begin{equation}\label{eq:twistedKE-vortex}
\left\{
\begin{array}{rcl}
iF_{h_{\varepsilon, 0}} + \frac{1}{2}(|\bm\phi|^2_{h_{\varepsilon, 0}} - \tau)\omega_{\varepsilon, 0} & = & -\sum_k \alpha_k\chi_k^{'\varepsilon} \\
\text{Ric } \omega_{\varepsilon,0} & =&
\frac{2\pi}{\text{Vol}_{\omega_0}} \tilde \chi\omega_{\varepsilon,0} + \sum_j (1-\beta_j)\chi_j^\varepsilon
\end{array}
\right.
\end{equation}

The second equation with $\sum_j(1-\beta_j)\chi_j^\varepsilon$ replaced by a general closed $(1,1)$-form is called \emph{twisted K\"ahler-Einstein equation} whose solution is called a \emph{twisted K\"ahler-Einstein metric}.  This was studied by Aubin and Yau in their classical continuity path to solve K\"ahler-Einstein problem, and this also appears in various other geometric situations, \cite{CDS, ST}. The specialization to this particular kind of ``twisting forms'' smoothing an integration current was used in the study of the existence and regularity of conical K\"ahler-Einstein metrics, for instance see \cite{GP, HX, Yao}. 

Given $\omega_0$ being the constant curvature metric on $\Sigma$, i.e. $\text{Ric }\omega_0=\frac{2\pi}{\text{Vol}_{\omega_0}}\chi \omega_0$ with $\chi=\chi(\Sigma)=2-2g(\Sigma)$.  For any other constant $\tilde\chi$, the way of obtaining twisted K\"ahler-Einstein metric with twisting form $\xi=\frac{2\pi}{\text{Vol}_{\omega_0}} (\chi - \tilde \chi)\omega_0+ i \partial\bar\partial F\in (\chi-\tilde\chi)[\omega_0]$, i.e. solution of 
\begin{equation}
\text{Ric }\omega = \frac{2\pi}{\text{Vol}_{\omega_0}} \tilde \chi  \omega + \xi
\end{equation}
is via solving the continuity path
\begin{equation}
\text{Ric }\omega = (1-t)\text{Ric }\omega_0 + t (\frac{2\pi}{\text{Vol}_{\omega_0}}\tilde \chi \omega + \xi), 
\end{equation}
which is equivalent to the PDE:

\begin{equation}
\omega_{u_t} = e^{-2t \frac{2\pi}{\text{Vol}_{\omega_0}} \tilde \chi u_t - tF}\omega_0. 
\end{equation}

It is proved by Aubin and Yau that for any $\tilde\chi\leq 0$ and $\xi$ of the above, the equation admits a unique solution.  Let us denote the twisted K\"ahler-Einstein metric by $\omega_{\varepsilon, 0}$ when $\xi=\sum_j (1-\beta_j)\chi_j^\varepsilon$.

The first equation of the system  $\eqref{eq:smooth-system}_{\varepsilon, \alpha}^{sm}$ (when fixing the underlying K\"ahler metric $\omega_{\varepsilon,0}$) is a vortex type equation, with a smooth twisting term.  In general, given a Riemann surface $\Sigma$ equipped with a K\"ahler metric $\omega$, letting $\eta$ be any smooth $(1,1)$-form on $\Sigma$ with $\int_\Sigma \eta=2\pi b$. Since any $(1,1)$-form on a Riemann surface is necessarily closed, we could write $\eta=\frac{2\pi}{\text{Vol}_{\omega}}b\omega+ i\partial\bar\partial F$ for some smooth function $F$, then we can consider the equation:
\begin{equation}\label{eq:twisted-vortex-equation}
iF_{h} + \frac{1}{2} (|\bm\phi|_h^2-\tau)\omega= \eta
\end{equation}

\begin{definition}[twisted vortices]
	\label{def:twisted-vortices}
	If $h$ satisfies the Equation \eqref{eq:twisted-vortex-equation}, then it is called a \emph{twisted vortices} with twist form $\eta$. 
	\end{definition}

The following is a necessary condition (obtained by integrating the equation over $\Sigma$) for the existence of twisted vortices with twist form $\eta$:
\begin{equation}\label{eq:necessary-sufficient}
\tau \cdot \frac{\text{Vol}_\omega}{2\pi} > 2(N-b).
\end{equation}

Define $\eta_t=\frac{2\pi}{\text{Vol}_{\omega}} b\omega+ ti\partial\bar\partial F$ for $t\in [0,1]$. We could solve the continuity path

\begin{equation}\label{eq:twisted-vortex}
iF_{h_t}+ \frac{1}{2} (|\bm\phi|_{h_t}^2 - \tau )\omega = \eta_t
\end{equation}

At $t=0$, this becomes the vortex equation 
\[
iF_h + \frac{1}{2}(|\bm\phi|_h^2-\tau)\omega
= 
\frac{2\pi}{\text{Vol}_{\omega}} b\omega
\]

studied by Noguchi, Bradlow and Garc\'ia-Prada, whose existence and uniqueness was established exactly under the numerical assumption \eqref{eq:necessary-sufficient}.  Let us denote the unique solution by
$h_0$, i.e. 

\begin{equation}
iF_{h_0} + \frac{1}{2} (|\bm\phi|_{h_0}^2-\tau)\omega
= 
\frac{2\pi}{\text{Vol}_{\omega}} b\omega
\end{equation} 

Write $h_t = h_0e^{2f_t}$, then Equation $\eqref{eq:twisted-vortex}_{t}$ is equivalent to the following PDE:

\begin{equation}\label{eq:twisted-vortex-PDE}
\Delta f_t + \frac{1}{2} |\bm\phi|_{h_0}^2 ( e^{2f_t}-1) = - \frac{1}{2} \Delta(tF)
\end{equation}

To get \emph{openness} of this continuity path, we set up a map between Banach spaces:
\begin{align}
\mathcal{V}: C^{2,\alpha}(\Sigma)\times [0,1] & \longrightarrow  C^\alpha(\Sigma)\\
(f, t) & \longmapsto \Delta f + \frac{1}{2}|\bm\phi|_{h_0}^2 (e^{2f} - 1) + \frac{1}{2}\Delta (tF)
\end{align}
At the solution $(f_t, t)$ to Equation \eqref{eq:twisted-vortex-PDE}, the linearization in the direction $(\dot f, 0)$ is:
\[
\delta\mathcal{V}|_{(f_t, t)}(\dot f) 
= \Delta \dot f + |\bm\phi|_{h_0}^2 e^{2f_t} \dot f
\]

The kernal is easily seen to be $\{0\}$ by integrating 
\begin{equation}
\label{eqn:integral-vanishing}
\int_\Sigma \left(  \dot f \Delta \dot f + |\bm\phi|_{h_t}^2 \dot f 
\right)\omega
=
\int_\Sigma \left( 2 |\nabla \dot f|^2 + |\bm\phi|_{h_t}^2 \dot f^2 \right)\omega
=0
\end{equation}

Therefore $\delta\mathcal{V}|_{(f_t,t)}$ is an invertible bounded linear operator from $C^{2,\alpha}(\Sigma)$ to $C^\alpha(\Sigma)$, and the openness follows from standard implicit function theorem in Banach spaces.  

To get \emph{closedness}, we need some a priori estimates. The first one is the upper bound of the Hermitian metrics $h_t$ which can be achieved by maximum principle, applied to $\Phi_t=|\bm\phi|_{h_t}^2$ satisfying the equation 
\begin{equation}
\label{eqn:maximum-principle}
\Delta \Phi_t
=
- 2 \frac{|\nabla \Phi_t|^2}{\Phi_t} 
+ \Phi_t \left(
\tau - \Phi_t 
+ 2b - \Delta (tF)
\right)
\end{equation}

\begin{lemma}For any $t\in [0,1]$, 
	\[
	|\bm\phi|_{h_t}^2 \leq \tau + 2b +  |\!|\Delta F|\!|_{C^0(\Sigma)}
	\]
\end{lemma}

According to this $C^0$ bound of $|\bm\phi|_{h_t}^2$, we know $\Delta f_t$ is uniformly bound by equation \eqref{eq:twisted-vortex-PDE}, therefore 
\[
|\!| f_t -  \frac{1}{\text{Vol}_{\omega}} \int_\Sigma f_t\omega|\!|_{C^0(\Sigma)}
\leq 
C
\]

Integrating Equation \eqref{eq:twisted-vortex-PDE}, we obtain
\[
\frac{1}{\text{Vol}_{\omega}} 
\int_\Sigma |\bm\phi|_{h_0}^2 e^{2f_t}\omega
= 
\frac{1}{\text{Vol}_{\omega}} \int |\bm\phi|_{h_0}^2 \omega
=
\frac{1}{\text{Vol}_{\omega}}\int_\Sigma e^{2f_t + \log |\bm\phi|_{h_0}^2 }\omega
\geq 
e^{\frac{1}{\text{Vol}_{\omega}} \int_\Sigma \left(  2f_t + \log |\bm\phi|_{h_0}^2 \right) \omega}
\]
from which we get 
\[
\max_\Sigma f_t \geq 0
\]
and 
\[
\frac{1}{\text{Vol}_{\omega}}\int_\Sigma 2f_t\omega
\leq 
\log \big( \frac{1}{\text{Vol}_{\omega}} \int_\Sigma |\bm\phi|_{h_0}^2 \omega  \big)
- \frac{1}{\text{Vol}_{\omega}} \int_\Sigma \log |\bm\phi|_{h_0}^2 \omega
:= C'
\]

It follows that 
\begin{lemma}\label{eq:potential-C0-twisted-vortex}
There exists $C>0$ independent of $t\in [0, 1]$ such that
\begin{equation*}
|\!|f_t|\!|_{C^0(\Sigma)}
\leq 
C
\end{equation*}
\end{lemma}

The standard $W^{2,p}$ estimate applied to Equation \eqref{eq:twisted-vortex-PDE} implies $|\!|f_t|\!|_{W^{2,p}(\Sigma)}$ is uniformly bounded and Sobolev's embedding then implies $|\!|f_t|\!|_{C^{1,\alpha}(\Sigma)}$ is uniformly bounded. A bootstrapping argument to \eqref{eq:twisted-vortex-PDE} thus implies 

\begin{proposition}
	For any integer $k\geq 0$, there exists constant $C_{k}>0$ independent of $t\in [0,1]$ such that 
	\[
	|\!| f_t |\!|_{C^k(\Sigma)}
	\leq 
	C_k. 
	\]
	\end{proposition}

In conclusion, we establish a generalization of the theorem of Noguchi, Bradlow and Garc\'ia-Prada \cite{Noguchi, Brad, Garcia-Prada} to the case of twisted vortices equation. 

\begin{theorem}[Existence and Uniqueness of twisted vortices]\label{thm:existence-uniqueness}
	Let $\eta$ be a smooth $(1,1)$-form on the compact Riemann surface $\Sigma$ with $\frac{1}{2\pi}\int_\Sigma \eta=b$, and $\bm\phi$ be a holomorphic section of the holomorphic line bundle $L$ with $c_1(L)=N$. Then, the twisted vortices equation \eqref{eq:twisted-vortex-equation} admits a solution iff the numerical assumption 
	$
	\tau\cdot \frac{\text{Vol}_\omega}{2\pi} > 2(N-b)
	$
	is satisfied. Moreover, in this case, the solution is unique. 
	\end{theorem}

\begin{proof}
Let $h$ and $h'=h e^{2f}$ be two solutions to \eqref{eq:twisted-vortex-equation}. Then by subtracting the two equations, we get 
	\[
	\Delta f + \frac{1}{2} |\bm\phi|_h^2 ( e^{2f}-1)=0
	\]
	The maximum principle does not directly apply since the maximum of $f$ might appear at the vanishing point of $\bm\phi$. 	However, we could still prove $f\equiv 0$ by using an easy perturbation argument. 
	
	We set up the map:
	\begin{align*}
	\mathcal{P}: C^{2,\alpha}(\Sigma)\times [0,1]
	& \longrightarrow C^\alpha(\Sigma)\\
	(f,\delta) 
	& \longmapsto \Delta f + \frac{1}{2} ( |\bm\phi|_h^2 + \delta) (e^{2f}-1)
	\end{align*}
	Since $\delta \mathcal{P}|_{(f,0)}(\dot f)= \Delta \dot f +  |\bm\phi|_h^2 e^{2f}\dot f$ is an invertible operator from $C^{2,\alpha}$ to $C^\alpha$, implicit function theorem tells us that for $\delta\in (0, \delta_0)$, there exists a unique $f_\delta$ such that 
	$
	\mathcal{P}(f_\delta, \delta)=0
	$
	and $|\!|f_\delta - f|\!|_{C^{2,\alpha}}\to 0$ as $\delta\to 0$. 
	
	Applying maximum principle to 
	\[
	\mathcal{P}(f_\delta, \delta)
	=
	\Delta f_\delta + \frac{1}{2}(|\bm\phi|_h^2 + \delta) (e^{2f_\delta}-1)=0
	\]
	at the maximum and minimum of $f_\delta$ on $\Sigma$, we obtain that 
	\[
	f_\delta \equiv 0, \;\; \forall \delta \in (0, \delta_0)
	\]
	As a consequence $f\equiv 0$. 
	\end{proof}

\begin{remark}
	This theorem does not put restriction on the genus of the Riemann surface.  
\end{remark}

Going back to Equation \eqref{eq:twistedKE-vortex}, if the parameters satisfies the numerical assumption 
\[
\tau \cdot \frac{\text{Vol}_{\omega_0}}{2\pi} 
> 
2\tilde N, 
\]
we obtain a solution $h_{\varepsilon, 0}$. The data $(\omega_{\varepsilon, 0}, h_{\varepsilon,0})$ will serve as the starting point for the main continuity path $\eqref{eq:rcdGV}_{\varepsilon,\alpha}$,  deforming the coupling constant $\alpha$. 

\section{Twisted gravitating vortex}\label{sect:main}

In this section, we study system when the RHS twisting forms are replaced by general closed $(1,1)$-forms $\eta$ and $\xi$.
Concretely speaking, we consider  the coupled equations: 



\begin{equation}\label{eqn:twisted-gravitating-vortex}
\left\{
\begin{array}{rcl}
	i F_h + \frac{1}{2} (|\bm\phi|_h^2 - \tau) \omega & =  & \eta\\
	\text{Ric }\omega - 2\alpha i\partial\bar\partial |\bm\phi|_h^2  
	+ \alpha\tau (|\bm\phi|_h^2 - \tau) \omega 
	& =  &
	\tilde c \omega + \xi
	\end{array}
	\right.
	\end{equation}
for real 2-forms $\eta= \sqrt{-1}\eta_{k\bar l}\mathrm{d}z^k\wedge \mathrm{d}\bar z^l=\sqrt{-1}\eta_{1\bar 1}\mathrm{d}z\wedge\mathrm{d}\bar z$, and $\xi= \sqrt{-1}\xi_{k\bar l}\mathrm{d}z^k\wedge\mathrm{d}\bar z^l=\sqrt{-1}\xi_{1\bar 1}\mathrm{d}z\wedge\mathrm{d}\bar z$. 

\begin{definition}[twisted gravitating vortex]\label{def:twisted-gravitating-vortices}
	A pair $(\omega,h)$ satisfying the system 
	$\eqref{eqn:twisted-gravitating-vortex}_{\alpha}$ is called \emph{twisted gravitating vortex} with twisting forms $\eta$ and $\xi$.  
	\end{definition}

Without loss of generality, we only focus on the K\"ahler class of volume $2\pi$, the other K\"ahler class parallels the arguments in this section. Still denote $\omega_0$ to be the constant curvature metric in this fixed K\"ahler class and $h_0$ to be the Hermitian metric of constant curvature on $L$. If $\eta\in b_\eta[\omega_0]$ and $\xi\in b_\xi[\omega_0]$ for real constant $b_\eta$ and $b_\xi$, then we could write
\begin{align*}
\eta
& = 
b_\eta \omega_0 + i\partial\bar\partial F_\eta\\
\xi
& = 
b_\xi \omega_0 + i\partial\bar\partial F_\xi
\end{align*}
Write $\omega=\omega_0+ 2i\partial\bar\partial u$ and $h= h_0e^{2f}= h_0 e^{-F_\eta} e^{2\tilde f}$.  The twisted gravitating vortex equations is equivalent to the semilinear elliptic PDE system:

\begin{equation}\label{eqn:twisted-GV-PDEs}
\begin{split}
\Delta \tilde f
+
\frac{1}{2} \left(\Phi -\tau \right) (1-\Delta u)
& = -( N - b_\eta) \\
\Delta u + e^{4\alpha\tau \tilde f - 2\alpha \Phi - 2\tilde c u}e^{-F_\xi} 
& = 1
\end{split}
\end{equation}
where $\Phi= |\bm\phi|_h^2 = |\bm\phi|^2 e^{-F_\eta} e^{2\tilde f}$. From now on, we assume $\eta\leq 0$ and $\xi\geq 0$, i.e. $\eta_{1\bar 1}\leq 0, \xi_{1 \bar 1}\geq 0$. Write $\tilde \chi = \chi - b_\xi$, $\tilde N= N - b_\eta$ and $\tilde c = \tilde \chi - 2\alpha\tau \tilde N$. We also assume $g(\Sigma)\geq 2$.  It will be clear the role played  by these assumptions in the openness and closedness argument. 


	\subsection{Openness}

	We are taking a more direct way here rather than the momentum map interpretation of  \cite{Al-Ga-Ga-P} studying the non twisted gravitating vortex equation.  For this part, we use the system \eqref{eqn:twisted-gravitating-vortex} instead of the PDE system \eqref{eqn:twisted-GV-PDEs}. 
	
	Fixing $k\geq 2$ be an integer. Let $L_p^2(\Sigma)$ be the space of functions on $\Sigma$ whose distributional derivatives up to order $p$ are square integrable. Denote $ \widehat{\mathcal{U}} \subset L_{k+4}^2(\Sigma)/\mathbb{R}$ to be the open neighborhood of $0$  consisting of $v$ such that $\omega=\omega_0+2i\partial\bar\partial v$ is a K\"ahler metric. Define a nonlinear differential operator
	\[
	\begin{array}{rll}
	\wp: \mathbb{R}\times L_{k+4}^2(\Sigma)\times \widehat{\mathcal{U}}
	& \longrightarrow &
	L_{k+2}^2(\Sigma)\times L_k^2(\Sigma)\\
	(\alpha, f, v)
	& \mapsto &
	(\wp_1(\alpha, f,v), \wp_2(\alpha, f,v))
	\end{array}
	\]
	
	where 
	\begin{align*}
	\wp_1(\alpha,f,v)
	& = 
	i\Lambda_\omega F_h +  \frac{1}{2}|\bm\phi|_h^2 - \frac{\tau}{2} - \Lambda_\omega \eta\\
	\wp_2(\alpha,f,v)
	&
	=   - S_\omega -\alpha \Delta_\omega |\bm\phi|_h^2 + 
	2\alpha\tau \Lambda_\omega (iF_h - \eta) +
	\tilde c + \Lambda_\omega \xi
	\end{align*}
	for $\omega=\omega_0+ 2i\partial\bar\partial v$ and $h=h_0 e^{2f}$. It is easily seen that $\wp$ is well-defined since $L_{k+4}^2(\Sigma)\subset C^0(\Sigma)$ by Sobolev embedding on $\Sigma$. This operator is order $2$ in $f$ and order $4$ in $v$. 
	
	\begin{proposition}
		\label{prop:kernal}
		The map $\wp$ is a $C^1$ map with Fr\'echet derivative (with respect to $(f,v)$) at $(\alpha,f,v)\in \mathbb{R}\times L_{k+4}^2(\Sigma)\times\widehat{\mathcal{U}}$ given by

		\[
		\begin{array}{rll}
		D\wp|_{(\alpha,f,v)}: L_{k+4}^2(\Sigma)\times L_{k+4}^2(\Sigma)/\mathbb{R}
		& \longrightarrow &
		L_{k+2}^2(\Sigma)\times L_k^2(\Sigma)\\
		(\dot f, \dot v)
		& \longmapsto &
		(		\delta \wp_1|_{(\alpha, f,v)} (\dot f, \dot v), 		\delta \wp_2|_{(\alpha, f,v)} (\dot f, \dot v))
		\end{array}
		\]
		
		where 
		\begin{equation}
		\label{eqn:linearization-operator}
		\begin{split}
		\delta \wp_1|_{(\alpha, f,v)} (\dot f, \dot v)
		& = 
		\Delta_\omega \dot f
		+ \frac{1}{2}(\tau - |\bm\phi|_h^2) \Delta_\omega \dot v
		+ \dot f |\bm\phi|_h^2
		+ \wp_1(\alpha, f,v) \Delta_\omega \dot v\; ; \\
		\delta \wp_2|_{(\alpha, f,v)}(\dot f, \dot v)
		& = 
		\frac{1}{2}\Delta_\omega^2 \dot v 
		- 
		\tilde c \Delta_\omega \dot v
		-
		2\alpha  \Delta_\omega ( \dot f|\bm\phi|_h^2) 
		+
		2\alpha\tau \Delta_\omega \dot f
		- \wp_2(\alpha, f,v) \Delta_\omega \dot v\; . 
		\end{split}
		\end{equation}

		In particular, suppose $\wp(\hat \alpha, \hat f, \hat v)=(0,0)$ with $\hat\alpha\geq 0$, then 
		$
		\ker D\wp|_{(\hat\alpha,\hat f, \hat v)}
		=
		\{(0,0)\}. 
		$
	\end{proposition}

	\begin{proof}	
		Direct calculation (for the clarity and consistency of notations we adopt the familiar conventions in K\"ahler geometry) shows that 
		\begin{equation}
		\label{eqn:linearization-operator}
		\begin{split}
		\delta \wp_1|_{(\alpha, f,v)} (\dot f, \dot v)
		& = 
		\Delta_\omega \dot f - 2F^{i\bar j} \dot v_{i\bar j} 
		+ 
		\dot f |\bm\phi|_h^2 
		+
		2  \eta^{i\bar j} \dot v_{i\bar j} \\
		\delta \wp_2|_{(\alpha,f,v)}(\dot f, \dot v)
		& = 
		\frac{1}{2} \Delta_\omega^2 \dot v 
		+ 2 R^{i\bar j}\dot v_{ij}
		- 2\alpha \Delta_\omega \left( \dot f|\bm\phi|_h^2 \right)
		- 4\alpha \dot v^{i\bar j} \partial_i\partial_{\bar j} |\bm\phi|_h^2 \\
		& \qquad 
		+ 2\alpha\tau \left( 
		\Delta_\omega \dot f 
		- 2 ( F^{i\bar j} - \eta^{i\bar j}) \dot v_{i\bar j}	\right) 
		- 2\xi^{i\bar j}\dot v_{i\bar j}
		\end{split}
		\end{equation}
		
		The formula in the proposition follows from the definition of $\wp_1, \wp_2$ then. 	Looking at the coefficients of the linear differential operators $\delta \wp_1, \delta \wp_2$ and their dependence on $(\alpha, f,v)$, it is clear they are continuous with respect to the variables $(\alpha, f, v)$ in $\Hom(L_{k+4}^2\times L_{k+4}^2/\mathbb{R}, L_{k+2}^2\times L_k^2)$ under the operator norm. The directional derivative of $\wp$ in the direction of $\alpha\in\mathbb{R}$ is even more clearly to be continuous with respect to $(\alpha, f, v)$. We conclude that $\wp$ is $C^1$ and its Fr\'echet derivative is as claimed.

		Suppose now $\wp(\hat\alpha, \hat f,\hat v)=0$. Write $\hat\omega=\omega_0+2i\partial\bar\partial \hat v$, $\hat h=h_0e^{2\hat f}$, the two directional derivatives are simplified as 
		\begin{equation}
		\label{eqn:linearization-first-second}
		\begin{split}
		\delta \wp_1:=
		\delta \wp_1|_{(\hat \alpha, \hat f,\hat v)}(\dot f, \dot v)
		& = 
		\Delta_{\hat\omega} \dot f
		+ 
		\frac{1}{2} (\tau - |\bm\phi|_{\hat h}^2) \Delta_{\hat\omega} \dot v
		+ \dot f |\bm\phi|_{\hat h}^2\\
		\delta\wp_2:=
		\delta \wp_2|_{(\hat \alpha, \hat f,\hat v)}(\dot f, \dot v)
		& = 
		\frac{1}{2}\Delta_{\hat\omega}^2 \dot v 
		- 
		\tilde c \Delta_{\hat \omega} \dot v
		-
		2\hat\alpha  \Delta_{\hat\omega} ( \dot f|\bm\phi|_{\hat h}^2) 
		+
		2\hat\alpha\tau \Delta_{\hat\omega} \dot f
		\end{split}
		\end{equation}
		
		Let $\eta_{\dot v}$ be the Hamiltonian vector field generated by $\dot v$ under the symplectic form $\omega$, i.e. $\eta_{\dot v}\lrcorner \omega = \mathrm{d}\dot v$.  
		Notice that $\eta_{\dot v}=- \sqrt{-1} \dot v^k \frac{\partial}{\partial z_k} + \sqrt{-1} \dot v^{\bar k} \frac{\partial}{\partial \bar  z_k}
		$ and 
		\begin{equation}
		\label{eqn:contraction-term}
		\begin{split}
		\mathrm{d}\dot f + \eta_{\dot v}\lrcorner \left( iF_{\hat h} - \eta\right)
		& = 
		\left( 
		\dot f_k +  \frac{1}{2}(\tau-|\bm\phi|_{\hat h}^2) \dot v_k  \right)\mathrm{d}z^k 
		+ 
		\left( 
		\dot f_{\bar k} + \frac{1}{2}(\tau-|\bm\phi|_{\hat h}^2) \dot v_{\bar k}  \right)\mathrm{d}\bar{z}^k 
		\end{split}
		\end{equation}
		from which we obtain 
		\begin{equation}
		\begin{split}
		4\hat\alpha \int_\Sigma \left| \mathrm{d}\dot f + \eta_{\dot v}\lrcorner \left( iF_{\hat h} - \eta\right) \right|^2 \hat\omega 
		& = 
		8\hat\alpha 
		\int_\Sigma \left(  
		\dot f_k + \frac{1}{2}(\tau - |\bm\phi|_{\hat h}^2) \dot v_k
		\right)
		\left(  
		\dot f^k + \frac{1}{2}(\tau - |\bm\phi|_{\hat h}^2) \dot v^k
		\right) \hat\omega\\
		& = 
		8\hat\alpha \int |\nabla^{1,0}\dot f|^2 
		+ 2\hat\alpha \int (\tau - |\bm\phi|_{\hat h}^2 )^2 |\nabla^{1,0}\dot v|^2 
		+ 4\hat\alpha 
		\int (\tau - |\bm\phi|_{\hat h}^2) \left( 
		\dot f_k \dot v^k + \dot f^k \dot v_k
		\right)\\
		& = 
		8\hat\alpha \int |\nabla^{1,0}\dot f|^2 
		+ 2\hat\alpha \int (\tau - |\bm\phi|_{\hat h}^2 )^2 |\nabla^{1,0}\dot v|^2 \\
		& \qquad  + 4 \hat \alpha 
		\int (\tau - |\bm\phi|_{\hat h}^2) \dot f \Delta_\omega \dot v
		+ 4\hat\alpha \int \dot f \left(
		\partial_k |\bm\phi|_{\hat h}^2 \dot v^k + 
		\partial^k |\bm\phi|_{\hat h}^2 \dot v_k
		\right)
		\end{split}
		\end{equation}	
		
		On the other hand, 
		\begin{equation}
        J\eta_{\dot v} \lrcorner \mathrm{d}_A \bm\phi - \dot f \bm\phi
		= J ( -\sqrt{-1}\dot v^k \partial_k + \sqrt{-1} \dot v^{\bar k} \partial_{\bar k})\lrcorner (D_i\bm\phi \mathrm{d}z^i) - \dot f \bm\phi
		= \dot v^k D_k\bm\phi - \dot f\bm\phi		
		\end{equation}
		and thus 
		\begin{equation}
		\begin{split}
		\int \left| J\eta_{\dot v}\lrcorner \mathrm{d}_A \bm\phi - \dot f \bm\phi\right|_{\hat h}^2 
		& = 
		\int \left| \dot v^k D_k\bm\phi - \dot f \bm\phi\right|_{\hat h}^2\\ 
		&  = 
		\int \dot f^2 |\bm\phi|_{\hat h}^2 
		-  \dot f \langle \bm\phi, \dot v^k D_k\bm \phi\rangle 
		- \dot f\langle \dot v^k D_k\bm\phi, \bm\phi\rangle
		+ |\mathrm{d}_A\bm\phi|^2 |\nabla^{1,0}\dot v|^2 
		\end{split}
		\end{equation} 
		where $|\dot v^k D_k\bm\phi|^2 = |\mathrm{d}_A\bm\phi|^2 |\nabla^{1,0}\dot v|^2$ in this particular case where $\Sigma$ is complex $1$-dimensional.  Now, using the formula for $\delta\wp_1$ and $\delta\wp_2$, we have
		
		\begin{equation}
		\begin{split}
		4\hat\alpha\langle \dot f, \delta\wp_1\rangle_{L^2}
		+ \langle \dot v, \delta \wp_2 \rangle_{L^2}
		& := 
		4\hat \alpha\int_\Sigma  \dot f (\delta \wp_1) \hat\omega 
		+ 
		\int_\Sigma \dot v (\delta \wp_2)\hat\omega\\
		& = 
		8\hat\alpha \int |\nabla^{1,0}\dot f|^2 
		+ 
		2\hat\alpha \int (\tau - |\bm\phi|_{\hat h}^2) \dot f \Delta_{\hat\omega} \dot v 
		+ 
		4\hat\alpha \int \dot f^2 |\bm\phi|_{\hat h}^2 \\
		&\qquad + 
		\frac{1}{2}\int (\Delta_{\hat\omega} \dot v)^2 
		- 2\tilde c \int |\nabla^{1,0}\dot v|^2  
		- 2\hat\alpha \int \dot f |\bm\phi|_{\hat h}^2 \Delta_{\hat\omega} \dot v 
		+ 2\hat\alpha \tau 
		\int \dot f \Delta_{\hat\omega} \dot v\\
		& = 
		8\hat\alpha \int |\nabla^{1,0} \dot f|^2 
		+ 4\hat\alpha \int (\tau - |\bm\phi|_{\hat h}^2)  \dot f \Delta_{\hat\omega} \dot v 
		+ 4\hat\alpha \int \dot f^2 |\bm\phi|_{\hat h}^2 \\
		& \qquad + \frac{1}{2}\int (\Delta_{\hat\omega} \dot v)^2 
		- 2\tilde c \int |\nabla^{1,0} \dot v|^2\\
		& = 
		4\hat\alpha \int \left| \mathrm{d}\dot f + \eta_{\dot v}\lrcorner (iF_{\hat h}-\eta) \right|^2 
		+ 
		4\hat\alpha \int \left| J\eta_{\dot v}\lrcorner \mathrm{d}_A\bm\phi - \dot f\bm\phi\right|_{\hat h}^2  \\
		& \qquad + 
		\int 
		\frac{1}{2}(\Delta_{\hat\omega} \dot v)^2 
		- 2 \tilde c |\nabla^{1,0}\dot v|^2 
		-2\hat\alpha (\tau -|\bm\phi|_{\hat h}^2)^2|\nabla^{1,0}\dot v|^2 
		- 4\hat\alpha |\mathrm{d}_A\bm\phi|^2 |\nabla^{1,0}\dot v|^2 \\
		& =
		4\hat\alpha \int \left| \mathrm{d}\dot f + \eta_{\dot v}\lrcorner (iF_{\hat h}-\eta)\right|^2 
		+ 
		\left| J \eta_{\dot v}\lrcorner \mathrm{d}_A \bm\phi - \dot f\bm\phi\right|_{\hat h}^2 \\
		& \qquad + 
		\int 
		\frac{1}{2}(\Delta_{\hat \omega} \dot v)^2 
		- 2 |\nabla^{1,0}\dot v|^2  \left(
		\tilde c + \hat\alpha (\tau-|\bm\phi|_{\hat h}^2)^2 + 2\hat\alpha |\mathrm{d}_A\bm\phi|^2
		\right) \\
		& = 
		4\hat\alpha \int \left| \mathrm{d}\dot f + \eta_{\dot v}\lrcorner (iF_{\hat h}-\eta)\right|^2 
		+ 
		\left| J \eta_{\dot v}\lrcorner \mathrm{d}_A \bm\phi - \dot f\bm\phi\right|_{\hat h}^2 \\
		&\qquad + 
		\int \frac{1}{2}(\Delta_{\hat \omega}\dot v)^2 - 2 S_{\hat \omega} |\nabla^{1,0}\dot v|^2 
		+ 
		2\Lambda_{\hat\omega} \left( \xi - 2\hat\alpha |\bm\phi|_{\hat h}^2 \eta \right) |\nabla^{1,0} \dot v|^2
		\end{split}
		\end{equation}
		where we use the formula for scalar curvature
		\begin{equation}
		\label{eqn:scalar-curvature-formula}
		S_{\hat \omega}
		= 
		2\hat\alpha |\mathrm{d}_A\bm\phi|^2 + \hat\alpha (\tau - |\bm\phi|_{\hat h}^2)^2 + \tilde c 
		+ \Lambda_{\hat \omega} \left( \xi-2\hat\alpha |\bm\phi|_{\hat h}^2 \eta \right)
		\end{equation}
		derived from Equation \eqref{eqn:twisted-gravitating-vortex} directly.  In the final step, we notice that 
		\begin{align*}
		\int 
		\frac{1}{2}(\Delta_{\hat\omega}\dot v)^2 
		- 2 S_{\hat\omega} |\nabla^{1,0} \dot v|^2 
		& = 
		\int 2 \dot v \dot v^{i\phantom{i}j}_{\phantom{i}i\phantom{j}j} 
		- 2 S \dot v_i \dot v^i\\
		& = 
		\int 2 \dot v \left( 
		\dot v^{ij}_{\phantom{ij}i} - R^{j}_{\phantom{j}k}\dot v^k
		\right)_j - 2 S \dot v_i \dot v^i\\
		& = 
		\int 2 \dot v v^{ij}_{\phantom{ij}ij}
		- 2\dot v R^{i\bar j}\dot v_{i\bar j} 
		- 2\dot v S_k \dot v^k 
		+ 2 S_i \dot v^i \dot v 
		- S\dot v \Delta_{\hat\omega} \dot v\\
		& = 
		\int 2 \dot v^{ij}\dot v_{ij} 
		\end{align*}
		
		This concludes the formula (compare with \cite[Equation (3.7)]{FPY})
		\begin{equation}
		\label{eqn:kernal-inner-product}
		\begin{split}
		4\hat\alpha\langle \dot f, \delta\wp_1\rangle_{L^2}
		+ \langle \dot v, \delta \wp_2 \rangle_{L^2}
		& = 
		4\hat\alpha |\!| \mathrm{d}\dot f + \eta_{\dot v}\lrcorner (iF_{\hat h}-\eta)|\!|^2  
		+ 4\hat\alpha |\!| J\eta_{\dot v}\lrcorner \mathrm{d}_A \bm\phi - \dot f\bm\phi|\!|^2 \\
		& \qquad+ 2 |\!| \bar\partial \nabla^{1,0}\dot v|\!|^2 
		+ 2 \langle \Lambda_{\hat\omega} (\xi - 2\hat\alpha |\bm\phi|_{\hat h}^2 \eta), |\nabla^{1,0}\dot v|^2\rangle
		\end{split}
		\end{equation}
		
		Assuming $\xi$ and $-\eta$ are both nonnegative $(1,1)$-forms, then $D \wp|_{(\hat \alpha, \hat f, \hat v)}(\dot f, \dot v)=(0,0)$  implies $\nabla^{1,0}\dot v$ is a holomorphic vector field on $\Sigma$ and thus $\dot v=const$ by the assumption $g(\Sigma)\geq 2$. In the case $\hat \alpha> 0$, the vanishing of the second term of the right hand side of Equation \eqref{eqn:kernal-inner-product} implies $J\eta_{\dot v}\lrcorner \mathrm{d}_A\bm\phi - \dot f \bm\phi=0$ which then forces $\dot f\equiv 0$. 
		In the case $\hat \alpha=0$, the second equation of \eqref{eqn:linearization-first-second} shows $\Delta_{\hat\omega} \dot v\equiv 0$, and the first equation reduces to the case of twisted vortices,  it follows from equation \eqref{eqn:integral-vanishing} that $\dot f\equiv 0$. 
		\end{proof}
	
	\begin{remark}
		\label{formula:scalar-curvature-twisted}
		More generally, at a general pair $(f, v)$, by replacing the left hand side of Equation \eqref{eqn:contraction-term} with $$\mathrm{d}\dot f + \eta_{\dot v}\lrcorner \left( iF_h - \eta - \wp_1\omega\right)$$ and replacing the formula \eqref{eqn:scalar-curvature-formula} of scalar curvature with 
		\[
		S_\omega
		= 
		2\alpha(\tau - |\bm\phi|_h^2) \wp_1 - \wp_2 
		+ \Lambda_\omega (\xi - 2\alpha |\bm\phi|_h^2 \eta) 
		+ 
		2\alpha |\mathrm{d}_A \bm\phi|^2 + \alpha (\tau - |\bm\phi|_h^2)^2 + \tilde c,
		\]
		we derive that the Fr\'echet derivative satisfies
		
		\begin{equation}
		\begin{split}
		4\alpha\langle \dot f,  \delta \wp_1|_{( \alpha, f, v)}(\dot f, \dot v)\rangle_{L^2}
		& + \langle \dot v, \delta \wp_1|_{(\alpha, f, v)}(\dot f, \dot v) \rangle_{L^2}\\
		& = 
		4\alpha |\!| \mathrm{d}\dot f + \eta_{\dot v}\lrcorner \left( iF_h-\eta - \wp_1\omega
		\right)|\!|^2  
		+ 4\alpha |\!| J\eta_{\dot v}\lrcorner \mathrm{d}_A \bm\phi - \dot f\bm\phi|\!|^2 \\
		& 
		\qquad+ 2 |\!| \bar\partial \nabla^{1,0}\dot v|\!|^2 
		+ 2 \langle \Lambda_\omega (\xi - 2\alpha |\bm\phi|_h^2 \eta), |\nabla^{1,0}\dot v|^2\rangle\\
		& 
		\qquad +   4\alpha \int_\Sigma\wp_1 \big(
		\dot f \Delta_\omega \dot v 
		+ 
		(\tau - |\bm\phi|_h^2) |\nabla^{1,0}\dot v|^2 
		\big)\omega
		- 
		\int_\Sigma\wp_2 (
		\dot v \Delta_\omega \dot v + 2|\nabla^{1,0} \dot v|^2
		)\omega\;.
		\end{split}
		\end{equation}
		This formula is not needed in this paper, however it may be useful in further studies.
		
	\end{remark}

	\begin{proposition}
		\label{prop:image}
		Suppose $\wp(\hat \alpha, \hat f, \hat v)=(0,0)$ with $\hat\alpha\geq 0$, then
		\[
		\Im D\wp|_{(\hat\alpha,\hat f,\hat v)}
		= \{
		(\mathfrak{q}_1, \mathfrak{q}_2)|
		\mathfrak{q}_1 \in L_{k+2}^2(\Sigma), 
		\mathfrak{q}_2\in \mathring{L}_k^2(\Sigma,\hat \omega)
		\}\;,
		\]
		where	$\mathring{L}_{k}^2(\Sigma,\hat\omega)=\{
		\mathfrak{q}\in L_k^2(\Sigma)| \int_\Sigma \mathfrak{q} \hat\omega=0
		\}$. 
	\end{proposition}
	
	\begin{proof}
		Given $\mathfrak{q}_1 \in L_{k+2}^2(\Sigma), 
		\mathfrak{q}_2\in \mathring{L}_k^2(\Sigma,\hat \omega)$, let
		$\mathfrak{p}\in L_{k+2}^2(\Sigma)$ satisfy $\int_\Sigma \mathfrak{p}\hat\omega=0$ and 
		$$
		\Delta_{\hat\omega}\mathfrak{p}
		=
		\mathfrak{q}_2. 
		$$
		Using Equation \eqref{eqn:linearization-first-second}, to prove the proposition, it suffices to find $(\dot f, \dot v)\in L_{k+4}^2\times L_{k+4}^2$ such that 
		\begin{equation}
		\label{eqn:transformed-equations}
		\begin{split}
		\Upsilon_1(\dot f, \dot v)
		& :=
		\Delta_{\hat \omega} \dot f
		+ 
		\left(
		|\bm\phi|_{\hat h}^2 
		-
		2\hat\alpha (\tau - |\bm\phi|_{\hat h}^2)^2 
		\right)\dot f
		+ \tilde c (\tau - |\bm\phi|_{\hat h}^2 ) \dot v
		 = 
		\mathfrak{q}_1 
		-
		(\tau - |\bm\phi|_{\hat h}^2) \mathfrak{p}\\
		\Upsilon_2(\dot f, \dot v)
		& :=
		\frac{1}{2}
		\Delta_{\hat \omega} \dot v
		- \tilde c \dot v
		+ 
		2\hat\alpha(\tau - |\bm\phi|_{\hat h}^2) \dot f
	  =
		\mathfrak{p}
		\end{split}
		\end{equation}
		The map $\Upsilon=(\Upsilon_1, \Upsilon_2): L_{k+4}^2\times L_{k+4}^2 
		\longrightarrow 
		L_{k+2}^2 \times L_{k+2}^2 	$ is a linear elliptic differential operator of order $2$, with $\ker \Upsilon=\{(0,0)\}$ according to Proposition \ref{prop:kernal}. By varying $\alpha\in [0,\hat \alpha]$, we can view $\Upsilon$ as a continuous family of Fredholm operators. At $\alpha=0$, the operator $\Upsilon$ is both injective and surjective (since $\Upsilon_1$ and $\Upsilon_2$ decouples thus could be solved independently because our assumption is $\tilde c<0$) and thus has index $0$. By the homotopy invariance of index, we conclude that Equation \eqref{eqn:transformed-equations} is uniquely solvable.  The solution $(\dot f, \dot v)$ satisfies the elliptic estimate 
		\begin{equation}
		\begin{split}
		|\!| \dot f|\!|_{L_{k+4}^2} 
		+
		|\!| \dot v|\!|_{L_{k+4}^2}
		& \leq 
		C 
		\Big(
		|\!| \mathfrak{q}_1 
		- (\tau - |\bm\phi|_{\hat h}^2) \mathfrak{p}|\!|_{L_{k+2}^2}
		+ 
		|\!|\mathfrak{p}|\!|_{L_{k+2}^2}
		\Big)\\
		& \leq 
		C\Big(
		|\!| \mathfrak{q}_1|\!|_{L_{k+2}^2}
		+ 
		|\!| \mathfrak{q}_2|\!|_{L_k^2}
		\Big). 
		\end{split}
		\end{equation}		
	\end{proof}

\begin{remark}
In this argument, we use the assumption $\tilde c<0$. The argument fails for the case $\tilde c=0$ and $\hat\alpha=0$. 
\end{remark}	
	\begin{proposition}[openness]
		\label{thm:openness}
		Suppose the twisted gravitating vortex equations \eqref{eqn:twisted-gravitating-vortex} admits a solution $(\hat\omega, \hat h)=(\omega_0+2i\partial\bar\partial \hat v, h_0e^{2\hat f})$ with $\hat \alpha\geq 0$, then there exists $\varepsilon>0$ such that for all $\alpha\in (\hat \alpha-\varepsilon, \hat\alpha+\varepsilon)$, there exists a unique pair  $(f_\alpha,v_\alpha)\in L_{k+4}^2(\Sigma)\times \widehat{\mathcal{U}}$ such that 
		$(\omega_\alpha,  h_\alpha)=(\omega_0+2i\partial\bar\partial v_\alpha, h_0e^{2 f_\alpha})$ satisfies the Equation \eqref{eqn:twisted-gravitating-vortex}, $(f_\alpha,v_\alpha)$ is $C^1$ in $\alpha$ with $(f_{\hat \alpha}, h_{\hat \alpha})=(\hat f, \hat v)$. 
	\end{proposition}
	
	\begin{proof}The modified map 
		\[
		\begin{array}{rll}
		\tilde\wp: \mathbb{R}\times L_{k+4}^2(\Sigma)\times \widehat{\mathcal{U}}
		& \longrightarrow &
		L_{k+2}^2(\Sigma)\times  \mathring{L}_k^2(\Sigma, \hat\omega)\\
		(\alpha, f, v)
		& \mapsto &
		(\wp_1(\alpha, f,v), \wp_2(\alpha, f,v)
		-\frac{1}{2\pi}\int_\Sigma \wp_2(\alpha, f, v)\hat \omega
		)
		\end{array}
		\]
		is $C^1$ with the same Fr\'echet derivative (with respect to $(f,v)$) at $(\hat\alpha, \hat f, \hat v)$ as $\wp$ since $\int_\Sigma \delta\wp_2 \hat\omega=0$.  By Proposition \ref{prop:kernal} and \ref{prop:image},  $D\tilde\wp|_{(\hat\alpha, \hat f, \hat v)}$ is an isomorphism between the Sobolev spaces 
		$L_{k+4}^2(\Sigma)\times \mathring{L}_{k+4}^2(\Sigma, \hat\omega)$ and $L_{k+2}^2(\Sigma)\times \mathring{L}_k^2(\Sigma, \hat \omega)$.
		Then, the implicit function theorem in Banach spaces \cite[Appendix]{DK} shows that there exists $\varepsilon>0$ and a $C^1$ path $\{(f_\alpha, v_\alpha)\in L_{k+4}^2(\Sigma)\times \mathring{L}_{k+4}^2(\Sigma, \hat\omega)| \alpha\in (\hat\alpha-\varepsilon, \hat\alpha+\varepsilon)\}$ such that $\tilde\wp(\alpha, f_\alpha, v_\alpha)=(0,0)$. This further implies $\wp_1(\alpha, f_\alpha, v_\alpha)=0$ and $\wp_2(\alpha, f_\alpha, v_\alpha)=const$. It follows then $\wp_2(\alpha, f_\alpha, v_\alpha)=0$ with $\tilde c=\tilde c(\alpha)= 2-2g(\Sigma) - 2\alpha\tau c_1(L)\cdot [\Sigma] - \frac{1}{2\pi}([\xi]-2\alpha\tau [\eta])\cdot [\Sigma]$.  
	\end{proof}


\subsection{A priori estimates and Closedness}
\label{sect:a priori}~

To derive useful a priori estimates, we mainly use the PDE system \eqref{eqn:twisted-GV-PDEs} in this section. Let $\Phi= |\bm\phi|_{h}^2$ as above, then there is a $C^0$ bound on $\Phi$ as in \cite[Lemma 4.1]{FPY}. 
	\begin{proposition}
		\label{prop:Phi-estimate}
		\[
		0\leq \Phi  \leq \tau
		\]
	\end{proposition}
	
	\begin{proof}
		The function $\Phi$ satisfies the equation 
		\[
		\Delta_{\omega} \Phi
		=
		- 2 \frac{|\nabla \Phi|^2}{\Phi} + \Phi(\tau - \Phi) + 2\Phi \Lambda_\omega \eta
		\]
		
		At the maximum of $\Phi$ (which is definitely positive), under the assumption $\eta\leq 0$ we could derive
		\[
		\Phi\leq \tau + 2\Lambda_\omega \eta\leq \tau. 
		\]
	\end{proof}

	Before we proceed to $C^0$ estimates of the potential functions $u$ and $\tilde f$, by applying Jensen's inequality to the second and the first equation of \eqref{eqn:twisted-GV-PDEs} we obtain the following integral estimates.

	\begin{lemma}[Integral estimate]\label{lem:integral-estimate}
		There holds
		\[
		\int_\Sigma 
		\big( 4\alpha\tau \tilde f - 2 \tilde c u \big) \omega_0
		\leq 
		4\pi \alpha \tau 
		+ 
		\int_\Sigma F_\xi\omega_0
		\]
		and 
		\begin{align*}
		\int_\Sigma \big(  (2+4\alpha \tau) \tilde f - 2 \tilde c u  \big) \omega_0
		 \leq 
		4\pi \alpha\tau + 2\pi \log (\tau - 2\tilde N) - \int_\Sigma \log |\bm\phi|^2 \omega_0 
		+\int_\Sigma F_\xi \omega_0	+\int_\Sigma F_\eta \omega_0\;.
		\end{align*}
	\end{lemma}
	
	Define $\alpha_*:=\frac{-\tilde \chi}{\tau(\tau-2\tilde N)}>0$, then for any $\alpha\in [0, \alpha_*]$, the following inequality 
	\begin{equation}
	\tilde c+\alpha\tau^2 \leq 0
	\end{equation}
	holds. 
	
	Writing $y = e^{4\alpha\tau \tilde f - 2\tilde cu}$ for simplicity, combining the two equations of the system \eqref{eqn:twisted-GV-PDEs}, we obtain
	\begin{equation}
	\Delta \log y
	=
	2y\left[
	\tilde c+\alpha\tau^2 -\alpha\tau \Phi
	\right]
e^{-2\alpha \Phi} e^{-F_\xi}
	-
	2(\tilde c+2\alpha\tau\tilde N)\;.
	\end{equation}
	
	By the Green's representation formula, 
	
	\begin{equation}
	\begin{split}
	\log y (P)
	& =
	\frac{1}{2\pi}\int_\Sigma (4\alpha\tau \tilde f- 2\tilde cu) \omega_0
	+
	\int_\Sigma  G(P, \cdot) \Delta \log y\;	\omega_0	\\
	& =
	\frac{1}{2\pi}\int_\Sigma (4\alpha\tau \tilde f- 2\tilde cu) \omega_0
	+
	\int_\Sigma 2 G(P, \cdot) 		
	\left\{  
	y\left[
	\tilde c+\alpha\tau^2 -\alpha\tau \Phi
	\right]
	e^{-2\alpha \Phi} e^{-F_\xi}
	-
	(\tilde c+2\alpha\tau\tilde N)
	\right\} \omega_0\\
	& \leq
	\frac{1}{2\pi}\int_\Sigma (4\alpha\tau \tilde f- 2\tilde c u) \omega_0
	- 2 (\tilde c+2\alpha\tau \tilde N) \int_\Sigma  G(P, \cdot) \omega_0\\
	& \leq 
	2\alpha\tau
	- 2 \tilde \chi \int_\Sigma  G(P, \cdot) \omega_0
	 + \frac{1}{2\pi} \int_\Sigma F_\xi \omega_0 \\
	 & \leq 
	 C + \frac{1}{2\pi} \int_\Sigma F_\xi \omega_0\\
	 & := C_1
	\end{split}
	\end{equation}

	Going back to the system, 
	
	\begin{equation}\label{eq:reducedConf}
	\Delta \tilde f
	=
	- \tilde N + \frac{1}{2}(\tau - \Phi )
	y e^{-2\alpha \Phi}e^{-F_\xi}
	\end{equation}
	and 
	\begin{equation}\label{eq:Kahlerpotential}
	\Delta u 
	= - y e^{-2\alpha \Phi}e^{-F_\xi}+1\;.
	\end{equation}
	
	With the uniform upper bound $0\leq y\leq e^{C}e^{\frac{1}{2\pi} \int_\Sigma F_\xi \omega_0}$ proved above, for any $p>1$ standard theory of $L^p$ estimate in PDE theory and Sobolev embedding theorem show that 
	
	\begin{equation}\label{eq:oscillation}
	\begin{split}
	|\!| \tilde f - \frac{1}{2\pi} \int_\Sigma \tilde f\; \omega_0 |\!|_{C^\gamma}
	& \leq 
	C(1+|\!| e^{-F_\xi} |\!|_{L^p} e^{\frac{1}{2\pi} \int_\Sigma F_\xi \omega_0})
		:=C_2 \\
	|\!| u - \frac{1}{2\pi} \int_\Sigma u\; \omega_0|\!|_{C^\gamma} 
	& \leq 
	C(1+|\!| e^{-F_\xi} |\!|_{L^p}e^{\frac{1}{2\pi} \int_\Sigma F_\xi \omega_0})
	=C_2
	\end{split}
	\end{equation}
for some $\gamma\in (0,1)$. 

	The consequence is that the 
	\begin{equation}
	\osc_\Sigma \log y \leq 2(4\alpha\tau - 2\tilde c) C_2 .
	\end{equation}
	
	Integrating the equation 
\[
\frac{1}{2\pi} \int_\Sigma  y e^{-2\alpha\Phi} e^{-F_\xi}\omega_0 =1
\]
	We obtain 
	\begin{equation}
	\max_\Sigma \log y
	\geq 
	\log \frac{1}{\frac{1}{2\pi} \int_\Sigma  e^{-2\alpha\Phi} e^{-F_\xi} \omega_0 }
	\geq 
	- \log \frac{1}{2\pi}{\int_\Sigma e^{-F_\xi} \omega_0}\;.
	\end{equation}
	
	In conclusion, there holds the following $C^0$ estimate of $\log y=4\alpha\tau f - 2\tilde c u$:
	\begin{equation}
	C_1
	\geq \log y
	\geq 
	- \log \frac{1}{2\pi}\int_\Sigma e^{-F_\xi}\omega_0
		-
		2 [4\alpha\tau - 2\tilde c]	C_2
		:= - C_3\;.
	\end{equation}

	Next, we would like to obtain $C^0$ bound of $\tilde f$ and $u$ respectively.

	Using Green's representation formula for $\tilde f$, we obtain 
	\begin{equation}
	\begin{split}
	\tilde f
	&	=
	\frac{1}{2\pi}\int_\Sigma \tilde f \omega_0
	+ \int_\Sigma G(P, \cdot) \Delta \tilde f\; \omega_0\\
	& = 
	\frac{1}{2\pi}\int_\Sigma \tilde f \omega_0
	+ \int_\Sigma G(P, \cdot) 
	\left\{ 
	- \tilde N + \frac{1}{2}(\tau - \Phi) y e^{-2\alpha\Phi}e^{-F_\xi}
	\right\} \omega_0\\
	& \leq  
	\frac{1}{2\pi}\int_\Sigma \tilde f \omega_0
	+ \frac{1}{2}\tau e^{C_1} |\!|G(P,\cdot)|\!|_{L^{p_*}} |\!| e^{-F_\xi}  |\!|_{L^p}\\
	& \leq 
	\frac{1}{2}C_3 
	+ \alpha\tau 
	+ 
	\frac{1}{2}\log (\tau - 2\tilde N) 
	- \frac{1}{4\pi} \int \log |\bm\phi|^2 \omega_0 
	+ \frac{1}{4\pi} \int_\Sigma (F_\xi + F_\eta) \omega_0 
	+ \frac{1}{2}\tau e^{C_1} |\!|G(P,\cdot)|\!|_{L^{p_*}} |\!| e^{-F_\xi}  |\!|_{L^p}\\
	&:=
	C_6
	\end{split}	
	\end{equation}
	where in the last inequality we used the upper bound of $\int_\Sigma \tilde f\omega_0$ derived from combining the uniform lower bound of $\log y$ and the second integral estimate in Lemma \ref{lem:integral-estimate} (and $\frac{1}{p_*}+ \frac{1}{p}=1$).

	By integrating \eqref{eq:reducedConf}, we obtain that 
	\begin{equation}
	\int_\Sigma
	|\bm\phi|^2 e^{-F_\xi - F_\eta} y e^{-2\alpha\Phi} e^{2\tilde f} \omega_0
	=
	2\pi(\tau - 2\tilde N)
	\end{equation}
	From this, we get 
	\begin{equation}
	\begin{split}
	\max \tilde f
	& \geq 
	\frac{1}{2} \log  2\pi(\tau - 2\tilde N) 
	- \frac{1}{2} \log \int_\Sigma |\bm\phi|^2 e^{-F_\xi-F_\eta} y e^{-2\alpha\Phi}\omega_0\\
	& \geq 
	\frac{1}{2} \log  2\pi(\tau - 2\tilde N) 
	- \frac{1}{2} (C_1+ \frac{1}{2\pi}\int_\Sigma F_\xi \omega_0) - \frac{1}{2} \log \sup |\bm\phi|^2 - \frac{1}{2} \log \int_\Sigma e^{-F_\xi-F_\eta} \omega_0\\
	& := - C_7
		\end{split}
	\end{equation}
	
	Combining the upper and lower bounds of $\max\tilde f$ with the oscillation bound \eqref{eq:oscillation}, we finally obtain 
	\begin{equation}
	- 2 C_2-C_7\leq 	\tilde f \leq C_6
	\end{equation}
	and 
	\begin{equation}
	- \frac{C_3 + 4\alpha \tau C_6}{-2\tilde c}
	\leq 
	u 
	\leq 
	\frac{C_1+ 4\alpha\tau (2C_2+C_7)}{-2\tilde c}\;.
	\end{equation}
	
	These $C^0$ estimates together with the oscillation bound \eqref{eq:oscillation} implies the following estimates:
	
	\begin{proposition}[$C^\gamma$ estimates]\label{prop:C0-estimate}
		For any $p>1$, there exists $\gamma=\gamma(p)>0$ and $C_8$ depending only on the upper bounds of $\int_\Sigma F_\xi \omega_0$,  $\int_\Sigma F_\eta\omega_0$ , $\int_\Sigma e^{-F_\xi - F_\eta}\omega_0$ and $|\!| e^{-F_\xi}|\!|_{L^p} $ such that 
		for any twisted gravitating vortex (i.e. solution to Equation \eqref{eqn:twisted-gravitating-vortex}) with $\alpha\in [0, \alpha_*]$ there holds
		\[
		|\!| u|\!|_{C^\gamma(\Sigma)}, |\!|\tilde f|\!|_{C^\gamma(\Sigma)}\leq \frac{C_8}{-2\tilde c}.
		\]
	\end{proposition}
	
	Regarding the solvability of Equation \eqref{eqn:twisted-gravitating-vortex}, the following existence and uniqueness theorem holds:
	
	\begin{theorem}[Existence and Uniqueness of twisted gravitating vortex]
		\label{thm:existence-uniqueness-twisted}
		Let $\Sigma$ be any compact Riemann surface with $g(\Sigma)\geq 2$, $\omega_0$ is a constant curvature K\"ahler metric on $\Sigma$. Let $L$ be a holomorphic line bundle equipped with a nonzero holomorphic section $\bm\phi$, and $-\eta, \xi$ be real nonnegative closed $(1,1)$-form with $\int_\Sigma \eta = \frac{2\pi}{\text{Vol}_{\omega_0}} b_\eta$ and $\int_\Sigma \xi= \frac{2\pi}{\text{Vol}_{\omega_0}} b_\xi$. Suppose $\tau\cdot \frac{\text{Vol}_{\omega_0}}{2\pi} > 2\tilde N$, and denote $\alpha_*=\frac{-\tilde \chi}{\tau(\tau-2\tilde N)}$. Then, for any coupling constant $\alpha\in [0, \alpha_*]$ there exists a unique solution $(\omega, h)$ to the twisted gravitating vortex equations \eqref{eqn:twisted-gravitating-vortex} with $\omega \in [\omega_0]$. 
	\end{theorem}

	\begin{proof}
		Now $\xi, \eta$ are fixed and only $\alpha$ varies in $[0,\alpha_*]$. 
			The standard Schauder estimate implies the $C^{2,\gamma}$ norm of $u$ and $\tilde f$ are bounded by $C_9$  independent of $\alpha\in [0,\alpha_*]$.  Then, the bootstrapping argument shows that all higher order norm of $u$ and $\tilde f$ are uniformly bounded (independent of $\alpha\in [0,\alpha_*]$). We thus get the closedness and then the existence of solution to Equation \eqref{eq:rcdGV}. 
		
		The uniqueness follows by connecting any twisted gravitating vortex with coupling constant $\alpha$ to the case when the coupling constant $\alpha=0$ via the continuity path studied, i.e. the system decouples as a twisted K\"ahler-Einstein metric and twisted vortices both proved to be unique. 
		\end{proof}

	\section{Singular gravitating vortices}
	We are ready to establish the existence of singular gravitating vortices introduced in section \ref{sect:twisted-vortices} using the previously proved a priori estimates. 
	
	\subsection{Existence}
	
	In this section, we specialize to a family of continuity paths $\eqref{eq:rcdGV}_{\varepsilon,\alpha}$ with $\eta_\varepsilon= - \sum_k \alpha_k \chi_\varepsilon^{'k}$ and $\chi_\varepsilon= \sum_j (1-\beta_j) \chi_\varepsilon^j$ for  $\varepsilon\in (0,1]$. For each $\varepsilon\in (0,1]$, the path could be solved for $\alpha\in [0,\alpha_*]$ by Theorem \ref{thm:existence-uniqueness-twisted}, whose solution is denoted by $(\omega_{\varepsilon,\alpha}, h_{\varepsilon,\alpha})$.  As $\varepsilon\to 0$, since the twisting forms $\eta_\varepsilon\rightarrow -2\pi \sum_k \alpha_k [r_k]$ and $\chi_\varepsilon\rightarrow 2\pi \sum_j (1-\beta_j) [q_j]$, we expect the twisted gravitating vortex $(\omega_{\varepsilon,\alpha}, h_{\varepsilon, \alpha})$ converges to ``solution" of Equation $\eqref{eq:cpGV}_\alpha$. 
	
	Since $F_{\eta_\varepsilon} = -\sum_k \alpha_k \log (|\bm t_k|^2 + \varepsilon), \; F_{\xi_\varepsilon}
	= \sum_j (1-\beta_j) \log (|\bm s_j|^2 + \varepsilon)$, it is clearly there exists $p>1$ depending only on $\beta_j$'s ($p$ can be taken as $\frac{1}{2} (1+ \min\{\frac{1}{1-\beta_1}, \cdots,  \frac{1}{1-\beta_M}\} )$ and $C>0$ independent of $\varepsilon\in (0,1]$  such that
	
	\begin{equation}
	   \begin{split}
	\int_\Sigma F_{\eta_\varepsilon} \omega_0
	& 
	\leq 
	- \sum_k\alpha_k  \int_\Sigma \log |\bm t_k|^2 \omega_0
	\leq C\\
	\int_\Sigma F_{\xi_\varepsilon} \omega_0
	& 
	\leq 
	\sum_j (1-\beta_j) \int_\Sigma \log (|\bm s_j|^2 + 1) \omega_0
	 \leq C\\
	\int_\Sigma e^{-F_{\xi_\varepsilon}  - F_{\eta_{\varepsilon}}} \omega_0
		& =
		\int_\Sigma \prod_j (|\bm s_j|^2 +\varepsilon)^{\beta_j-1} \prod_k (|\bm t_k|^2 + \varepsilon)^{\alpha_k} \omega_0\\
		& \leq 
		\int_\Sigma  
		 \prod_j |\bm s_j|^{2\beta_j-2} \prod_k (|\bm t_k|^2+1)^{\alpha_k}\omega_0
		 \leq C \\
	\int_\Sigma e^{-pF_{\xi_\varepsilon}}\omega_0
	& = 
	\int_\Sigma \prod_j |\bm s_j|^{2p(\beta_j-1)}\omega_0
	\leq C	 
	\end{split}
	\end{equation}
	
	With these integral upper bounds, Proposition \ref{prop:C0-estimate} implies 
	\begin{equation}
	|\!| u_{\varepsilon,\alpha} |\!|_{C^\gamma(\Sigma)} 
	+
	|\!| \tilde f_{\varepsilon,\alpha} |\!|_{C^\gamma(\Sigma)}
	\leq 
	C
	\end{equation}
	for some $\gamma>0$ and $C$ independent of $\varepsilon\in (0,1]$ and $\alpha\in [0,\alpha_*]$.

	\begin{proposition}
		For any $K\subset\subset\Sigma\backslash\{q_1, \cdots, q_M\}$, there exists $C_K>0$ independent of $\varepsilon\in (0,1]$ and $\alpha\in [0,\alpha_*]$ such that 
		\[
		|\!| u_{\varepsilon,\alpha}|\!|_{C^{2,\gamma}(K)} , |\!| \tilde f_{\varepsilon,\alpha}|\!|_{C^{2,\gamma}(K)}
		\leq 
		C_K. 
		\]
	\end{proposition}
	
	\begin{proof}
		
		We notice that $\Phi_{\varepsilon,\alpha}=|\bm\phi|^2 \prod_k (|\bm t_k|^2+\varepsilon)^{\alpha_k} e^{2\tilde f_{\varepsilon,\alpha}}$ has a uniform $C^\gamma$ bound on $\Sigma$ (possibly by redefining $\gamma$ to be $\min \{\gamma, 2\alpha_1, \cdots, 2\alpha_S \} $). For any $K\subset\subset\Sigma\backslash\{q_1, \cdots, q_M\}$, the term  $\frac{e^{4\alpha\tau \tilde f_{\varepsilon,\alpha} - 2\alpha \Phi_{\varepsilon,\alpha} - 2\tilde c u_{\varepsilon,\alpha} }}{\prod_j (|\bm s_j|^2+\varepsilon)^{1-\beta_j} }$ has a uniform $C^\gamma$ bound depending only on $K$ (independent of $\varepsilon$ and $\alpha\in [0,\alpha_*]$), and then the proposition follows from the standard Schauder's estimates. Precisely speaking, the second equation of the system \eqref{eq:smooth-system} implies $u_{\varepsilon,\alpha}$ has uniform $C^{2,\gamma}$ bound on $K$. Back to the first equation, we conclude $\tilde f_{\varepsilon,\alpha}$ has uniform $C^{2, \gamma}$ bound on $K$. 	
	\end{proof}
	
	For higher order estimates, we have 
	
	\begin{proposition}
		For any $K\subset\subset\Sigma\backslash\{q_1, \cdots, q_M; r_1, \cdots, r_S\}$ and integer $s\geq 3$, there exists $C_{K,s}>0$ independent of $\varepsilon\in (0,1]$ and $\alpha\in [0,\alpha_*]$ such that 
		\[
		|\!| u|\!|_{C^{s,\gamma}(K)} , |\!| \tilde f|\!|_{C^{s,\gamma}(K)}
		\leq 
		C_{K, s}
		\]
	\end{proposition}
	
	\begin{proof}
		We only need to notice that the for any integer $l\geq 1$, the $C^{l,\gamma}$ norm of $|\bm\phi|^2 \prod_k (|\bm t_k|^2+\varepsilon)^{\alpha_k}$ on $K$ is uniformly bounded. This fact together with the above proposition yield the proposition by a standard bootstrapping argument of Schauder's estimate. 
	\end{proof}

	Fixing any $\alpha\in (0, \alpha_*]$, we take a sequence $K_m\subset\subset\Sigma\backslash\{q_1,\cdots, q_M; r_1,\cdots, r_S\}$ with $K_m\subset K_{m+1}$ and $\cup_{m=1}^{+\infty} K_m= \Sigma\backslash\{q_1,\cdots, q_M; r_1,\cdots, r_S\}$. Applying Arzela-Ascoli's theorem successively on $K_m$, we get a subsequence $\varepsilon_m\to 0$ such that $(\omega_{\varepsilon_m,\alpha}, h_{\varepsilon_m,\alpha})$ converges to a limit $(\omega_{0,\alpha}, h_{0,\alpha})$ solving Equation \eqref{eq:cpGV} on   $\Sigma\backslash\{q_1,\cdots, q_M; r_1,\cdots, r_S\}$, with the potential functions $\tilde f_{0,\alpha}, u_{0,\alpha}\in C^\gamma(\Sigma)\cap C^\infty(\Sigma\backslash\{q_1,\cdots, q_M; r_1,\cdots, r_S\})$. The pair satisfies 
	
	\begin{equation}
	\begin{split}
	\Delta \tilde f_{0,\alpha}
	+
	\frac{1}{2} ( |\bm\phi|_{h_{0,\alpha}}^2 -\tau ) (1-\Delta u_{0,\alpha})
	& = -\tilde N\\
	\Delta u_{0,\alpha} + \frac{e^{4\alpha\tau \tilde f_{0,\alpha} - 2\alpha |\bm\phi|_{h_{0,\alpha}}^2 - 2\tilde c u_{0,\alpha}}}{\prod_j |\bm s_j|^{2(1-\beta_j)}} 
	& = 1
	\end{split}
	\end{equation}
	with $|\bm\phi|_{h_{0,\alpha}}^2 = |\bm\phi|^2 \prod_k |\bm t_k|^{2\alpha_k} e^{2\tilde f_{0,\alpha}}$ on $\Sigma\backslash\{q_1,\cdots, q_M; r_1,\cdots, r_S\} $. Therefore, $h_{0,\alpha} = h_0 \prod_k |\bm t_k|^{2\alpha_k} e^{2\tilde f_{0,\alpha}}$ is a H\"older continuous Hermitian metric on $L$ with parabolic singularity of order $2\alpha_k$ at $r_k$ for $k=1,2,\cdots,S$; and $\omega_{0,\alpha}=(1-\Delta u_{0,\alpha})\omega_0 = \prod_j |\bm s_j|^{2(\beta_j-1)} \digamma \omega_0$ is a H\"older continuous K\"ahler metric with conical singularity of angle $2\pi \beta_j$ at $q_j$ for $j=1,2,\cdots, M$. Moreover, $\omega_{0, \alpha}$ and $h_{0,\alpha}$ are smooth on $\Sigma\backslash\{q_1, \cdots, q_M; r_1, \cdots, r_S\}$.

    In conclusion, regarding the solvability of \eqref{eq:cpGV}, we finally reach the following existence theorem. 
	\begin{theorem}[Existence of singular gravitating vortices]
		\label{thm:existence-singular-gravitating-vortices}
		Let $\Sigma$ be any compact Riemann surface with $g(\Sigma)\geq 2$, $\omega_0$ is a constant curvature K\"ahler metric on $\Sigma$. Let $L$ be a holomorphic line bundle equipped with a nonzero holomorphic section $\bm\phi$. 
		Let $\beta_j\in (0,1)$ for $j=1,\cdots,M$ and $\alpha_k>0$ for $k=1,\cdots, S$ such that $\tilde \chi=\chi(\Sigma) - \sum_j (1-\beta_j)<0$, and $\tau \cdot \frac{\text{Vol}_{\omega_0}}{2\pi} >2\tilde N=2(N+\sum_k \alpha_k)$ and denote $\alpha_*=\frac{-\tilde \chi}{\tau(\tau-2\tilde N)}$. For any coupling constant $\alpha\in [0, \alpha_*]$, there exists a solution $(\omega, h)$ to Equation \eqref{eq:cpGV} with $\omega\in [\omega_0]$. Moreover, $\omega$ has conical singularity with angle $2\pi \beta_j$ at $q_j$ for $j=1,2,\cdots, M$ and $h$ is a Hermitian metric having parabolic singularities of order $2\alpha_k$ at $r_k$ for $k=1,2,\cdots, S$.
	\end{theorem}

	\subsection{Regularity}

	
	Continuing the notations introduced in the introduction, the following proposition shows how to build smooth gravitating vortices from singular gravitating vortices when the orders $\bm \alpha=(\alpha_1,\cdots, \alpha_S)$ and $\bm \beta=(\beta_1,\beta_2,\cdots, \beta_M)$ take some special values.
	
	\begin{proposition} 
		\label{prop:branched-cover}
		Let $\pi: \Sigma'\to \Sigma$ be a branched cover with degree $n_j$ around $q_j$, let $(\omega, h)$ be a singular gravitating vortices on $(\Sigma, L)$ with conical singularity with angle $2\pi\frac{1}{n_j}$ at $q_j \; (j=1,2,\cdots, M)$ for $\omega$ and  parabolic singularity of order $2m_k$ at $r_k\; (k=1,2,\cdots, S)$ for $h$, with the Higgs field $\bm\phi$. Then $\big( \pi^*\omega, \pi^*(h\otimes \prod_k h_k^{m_k}) \big)$ is a smooth gravitating vortex solution on $\big( \Sigma', \pi^*(L\otimes \prod_k L_{[r_k]}^{\otimes m_k}) \big)$, with the Higgs field $\pi^*\left(  \bm\phi\otimes \prod_k \bm t_k^{m_k} \right)$. 
	\end{proposition}

	\begin{proof}
		Obviously, for each $q_j$ where $\omega\sim \sqrt{-1}|z|^{2\beta_j-2}\mathrm{d}z\wedge\mathrm{d}\bar z$ and the branch map $\pi$ is modeled on the map $\pi(w)=w^{n_j}$, we have  $\pi^*\omega\sim \sqrt{-1}\mathrm{d}w\wedge\mathrm{d}\bar w$.  Thus, $\omega'=\pi^*\omega$ is a H\"older continuous K\"ahler metric on $\Sigma'$ which is bounded above and below by a background smooth K\"ahler metric $\tilde\omega_0$. The Hermitian metric $h'$ upstairs is a H\"older continuous non-degenerated Hermitian metric on $L'$. Moreover, $(\omega', h')$ satisfy the equation 
		\begin{equation}\label{eq:pull-back-cpGV}
		\left\{
		\begin{array}{rcl}
		iF_{h'} + \frac{1}{2}(|\bm\phi'|^2_{h'} - \tau)\omega' & = & 0\\
		\text{Ric } \omega' - 2\alpha i\partial\bar\partial |\bm\phi'|_{h'}^2 + \alpha\tau (|\bm\phi'|_{h'}^2 -\tau)\omega' & =& \tilde c\omega'
		\end{array}
		\right.
		\end{equation}
		on $\Sigma'$. If we write $\omega'=\omega_0'+2i\partial\bar\partial u'$ and $h'= h_0' e^{2f'}$ with $\omega_0'$ and $h_0'$ the background smooth constant curvature metrics, then $(u',f')$ satisfy the PDE system:
		\begin{equation}
		\begin{split}
		\Delta_{\omega_0'}  f'
		+
		\frac{1}{2} ( |\bm\phi'|_{h'}^2 -\tau ) (1-\Delta_{\omega_0'} u')
		& = -\tilde N\\
		\Delta_{\omega_0'} u' + e^{4\alpha\tau  f' - 2\alpha |\bm\phi'|_{h'}^2 - 2\tilde c u'} 
		& = 1
		\end{split}
		\end{equation}
		on $\Sigma'\backslash \pi^{-1} \left( \{q_1, \cdots, q_M, r_1, \cdots, r_S \} \right)$. 
        Standard elliptic regularity and bootstrapping arguments show that $u'$ and $f'$ are actually smooth on $\Sigma'$. 
		
	\end{proof}

\section{Singular Einstein-Bogomol'nyi equations}\label{sect:Bogmol'nyi-phase}

The case $\tilde c=0$, equivalently when the symmetry breaking parameter $\tau=\frac{\tilde \chi}{2\alpha \tilde N}$,  which is called Bogomol'nyi phase,  has particular interests in physics since it leads to non-vacuum spacetime satisfying  Einstein Field Equation with singularity and with pure magnetic field.  We call this system \emph{singular Einstein-Bogomol'nyi equations}.  This case is treated separately since the continuity method deforming coupling constant $\alpha$ does not work, the openness fails due to the nontrivial kernal of $D\wp $ at $\hat\alpha=0$ and the $C^0$ estimate is not clear along the path. 

%

The two equations in \eqref{eq:cpGV} can be combied into one elliptic PDE with a peculiar-looking nonlinear term:

\begin{equation}
\Delta \tilde f
+ 
\frac{1}{2} \lambda\left(
e^{2\tilde f 
	+\log |\bm \phi|^2
	+ \sum_k \alpha_k \log |\bm t_k|^2 }
- \tau \right)
e^{4\alpha\tau \tilde f - 2\alpha 
	e^{2\tilde f 
		+\log |\bm \phi|^2
		+ \sum_k \alpha_k \log |\bm t_k|^2 }
	- \sum_j (1-\beta_j)  \log |\bm s_j|^2}  
=
-\frac{2\pi \tilde N}{\text{Vol}_{g_0}}
\end{equation}
in which $\lambda$ is an indefinite positive constant. Notice that our $\Delta=\Delta_{g_0}$ is precisely the \emph{negative} of the Laplacian used by Yang \cite{Yang, HS}.  

When the singularities are absent, Yang used sup/subsolution method to find solution of this equation, cf. \cite[Equation (36)]{Yang} (see also \cite{HS}). We will show that Yang's method directly carries over to the current singular setting. The sets
\begin{align}
\mathcal{Z}
& =
\{(p_i,n_i)|p_i \text{ is a zero of }\bm\phi \text{ with multiplicity }n_i, i=1,2,\cdots, Z\}\\
\mathcal{C} 
& = 
\{ (q_j, 1-\beta_j)| j=1,2,\cdots, M\}\\
\mathcal{P}
& = 
\{
(r_k, \alpha_k)| k=1,2,\cdots, S
\}
\end{align}
are intended to represent the zeros of the Higgs field with corresponding multiplicities, the conical points of the K\"ahler metric with corresponding angle defects and the parabolic points of the Hermitian metric with corresponding degrees respectively.

 Write 
\[
u_0 
=
\log |\bm \phi|^2 + \sum_k \alpha_k \log |\bm t_k|^2,
\]
then the equation is written as 
\begin{equation}
\label{eq:singular-Bogomolnyi}
\begin{split}
\Delta \tilde f 
+ 
\frac{1}{2} \lambda e^{-\left( 2\alpha\tau u_0 + \sum_j (1-\beta_j) \log |\bm s_j|^2 \right) } F( 2\tilde f + u_0 ) 
& =
- \frac{2\pi \tilde N}{\text{Vol}_{g_0}}
\end{split}
\end{equation}
where $F(t)= e^{2\alpha\tau t -2\alpha e^t} (e^t - \tau)$. Suppose we  can find a solution to \eqref{eq:singular-Bogomolnyi} for some $\lambda>0$, then the conformal metric
\begin{align}
\label{def:conformal-metric}
g
=
\lambda \frac{e^{4\alpha\tau \tilde f - 2\alpha |\bm \phi|_h^2}}{\prod_j |\bm s_j|^{2(1-\beta_j)}}  g_0
\end{align}
where $|\bm \phi|_h^2=e^{2\tilde f+u_0} $ together with the Hermitian metric 
\begin{align}
\label{def:conformal-Hermitian}
h
=
h_0 \prod_k |\bm t_k|^{2\alpha_k} e^{2\tilde f}
\end{align}
satisfy the \emph{singular Einstein-Bogomol'nyi equations}:
\begin{equation}
\label{eq:Bogomolnyi-phase}
\begin{split}
iF_h 
+ 
\frac{1}{2} (|\bm \phi|_h^2 - \tau )\text{dvol}_g  
& =
-2\pi \sum_k \alpha_k [r_k]\\
S_g
+ 
\alpha (\Delta_g + \tau )(|\bm \phi|_h^2 - \tau) 
& = 2\pi \sum_j (1-\beta_j)\delta_{q_j}.
\end{split}
\end{equation}

Let us try to solve a regularization of Equation \eqref{eq:singular-Bogomolnyi} first. Let $u_0^\delta= \log (|\bm \phi|^2 + \delta) 
+ 
\sum_k \alpha_k \log (|\bm t_k|^2 + \delta)$, then the aim is to find solution $\tilde f$ to

\begin{equation}
\label{eq:combined}
\Delta \tilde f 
+ 
\frac{1}{2} \lambda e^{-\left( 2\alpha\tau u_0^\delta + \sum_j (1-\beta_j) \log \left( |\bm s_j|^2 + \delta\right) \right) } F( 2\tilde f + u_0^\delta ) 
=
- \frac{2\pi \tilde N}{\text{Vol}_{g_0}}
\end{equation}
for all $\delta\in (0,1)$ for some fixed $\lambda>0$. 

Let $v_0^\delta=2\alpha\tau u_0^\delta + \sum_j  (1-\beta_j) \log (|\bm s_j|^2 + \delta)$ be a regularization of $v_0=2\alpha\tau u_0 + \sum_j  (1-\beta_j)\log |\bm s_j|^2$ and 
$C_\delta
: =
1+\lambda \sup_{x\in \Sigma} \{  e^{-v_0^\delta(x)} \}\sup_{t\in \mathbb{R}} F'(t)$, and define a sequence of functions $\{\tilde f_n\}_{n=1,2,\cdots}$ iteratively by 

\begin{equation}
\label{eq:iterative}
\begin{split}
(\Delta + C_\delta) \tilde f_n
& =
-\frac{1}{2}\lambda e^{-v_0^\delta} F(2 \tilde f_{n-1}+u_0^\delta)
+ 
C_\delta \tilde f_{n-1}
- \frac{2\pi \tilde N}{\text{Vol}_{g_0}}\\
2\tilde f_1
& := 
- u_0^\delta + \log \tau.
\end{split}
\end{equation}
Notice that $\tilde f_n$ could be solved uniquely since $\Delta +C_\delta$ is a positive operator, and $\tilde f_n$ are all smooth on $\Sigma$ by elliptic regularity. Let $\Psi_\sigma$ be a $C^\infty $ cut-off function supported in $\cup_{x\in \mathcal{S}} B_{g_0}(x,2\sigma)$ and equals to $1$ on $\cup_{x\in \mathcal{S}} B_{g_0}(x, \sigma)$ where $\mathcal{S}=\mathcal{Z}\cup \mathcal{C}\cup \mathcal{P}$. Let $w$ solve
\[
\Delta w
=
-\frac{4\pi\tilde N}{\text{Vol}_{g_0}} \Psi_\sigma + C(\sigma)
\]
where $C(\sigma) = \frac{4\pi\tilde N}{\text{Vol}_{g_0}^2}\int_\Sigma \Psi_\sigma \text{dvol}_{g_0}$. If $\sigma $ is chosen sufficiently small, then $C(\sigma)>0$ is sufficiently small and ( by subtracting a big constant from $w$ to make $2w+u_0^\delta<\log \tau$ for all $\delta\in (0,1)$) there holds on $\cup_{x\in S} B_{g_0}(x, \sigma)$
\begin{equation}
\label{eq:supsolution}
\Delta w
< 
- \frac{ 2\pi\tilde N}{\text{Vol}_{g_0}}
- 
\frac{1}{2} \lambda e^{-v_0^\delta} F(2w+u_0^\delta) 
\end{equation}
for any $\lambda>0$ and $\delta\in (0,1)$. With $w$ and $\sigma$ chosen as above, we choose $\lambda>0$ sufficiently large such that 
\[
- 
\frac{1}{2}\lambda e^{-v_0^\delta} F(2w+u_0^\delta) 
=
\frac{1}{2}\lambda (\tau - e^{2w+u_0^\delta}) e^{-v_0^\delta}
\geq 
\frac{1}{2}\lambda (\tau - e^{2w+u_0^\delta}) e^{-v_0^1}>
\frac{2\pi \tilde N}{\text{Vol}_{g_0}} + \Delta w
\]
on $\Sigma\backslash \cup_{x\in \mathcal{S}} B_{g_0}(x, \sigma)$ for all $\delta \in (0,1)$. Therefore, the inequality \eqref{eq:supsolution} is satisfied on whole $\Sigma$, meaning that we have constructed a supsolution $w$ to Equation \eqref{eq:combined}. 

\begin{proposition}
	For any $\delta\in (0,1)$, the sequence of functions $\tilde f_n$ satisfies 
	\[
	\tilde f_1>  \tilde f_2> \cdots> \tilde f_n> \cdots >  w	
	\]
\end{proposition}

\begin{proof}
	By a direct calcuation
	\begin{equation*}
	\Delta \tilde f_1
	=
	-  \frac{2\pi N}{\text{Vol}_{g_0}} \frac{|\bm \phi|^2}{|\bm \phi|^2+\delta}
	- 
	\sum_k \frac{2\pi\alpha_k }{\text{Vol}_{g_0}} \frac{|\bm t_k|^2}{|\bm t_k|^2 + \delta} 
	+ 
	\frac{\delta|\bm\phi|^2}{(|\bm \phi|^2+\delta)^2} \left|\nabla^{1,0}\log |\bm \phi|^2\right| 
	+
	\sum_k \alpha_k \frac{\delta|\bm t_k|^2}{(|\bm t_k|^2+\delta)^2} \left|\nabla^{1,0}\log |\bm t_k|^2\right|
	\end{equation*}
The function $\tilde f_2-\tilde f_1$ satisfies 
	\begin{equation}
	\begin{split}
	\left( \Delta + C_\delta \right) \left( \tilde f_2-\tilde f_1\right)
	& =
	-  \frac{2\pi N}{\text{Vol}_{g_0}} \frac{\delta}{|\bm \phi|^2+\delta}
	- 
	\sum_k \frac{2\pi\alpha_k }{\text{Vol}_{g_0}} \frac{\delta}{|\bm t_k|^2 + \delta} \\
   &\qquad	-
	\frac{\delta|\bm\phi|^2}{(|\bm \phi|^2+\delta)^2} \left|\nabla^{1,0}\log |\bm \phi|^2\right| 
	-
	\sum_k \alpha_k \frac{\delta|\bm t_k|^2}{(|\bm t_k|^2+\delta)^2} \left|\nabla^{1,0}\log |\bm t_k|^2\right|\\
	& <0
	\end{split}
	\end{equation}
	The maximum principle concludes that $
	\tilde f_2 -\tilde f_1<0
	$ on $\Sigma$. On the other hand, there holds away from $\mathcal{Z}\cup \mathcal{P}$ that
	\begin{equation}
	\begin{split}
	(\Delta + C_\delta) \left(
	w-\tilde f_2\right)
	&< 
	C_\delta (w-\tilde f_1)
	-\frac{1}{2} \lambda e^{-v_0^\delta} \left(  
	F(2w+u_0^\delta) - F(2\tilde f_1 + u_0^\delta)\right)\\
	& = 
	\left( 
	C_\delta - \lambda e^{-v_0^\delta} F'(u_0^\delta+\xi) 
	\right) \left( w-\tilde f_1 \right)\\
	& <0
	\end{split}
	\end{equation}
	where $\xi$ is a function on $\Sigma$. The maximum principle implies that 
	\[
	w< \tilde f_2. 
	\]
	The claimed inequality follows inductively. Precisely speaking, the maximum principle applied to 
	\begin{equation}
	\begin{split}
	(\Delta + C_\delta)\left(
	\tilde f_{n+1}- \tilde f_n\right)
	& =
	\left(C_\delta -  \lambda e^{-v_0^\delta}F'(u_0+\kappa_n)
	\right) \left( \tilde f_{n} - \tilde f_{n-1} \right)\\
	(\Delta + C_\delta) \left(
	w-\tilde f_{n+1}\right)
	&< 
	\left( 
	C_\delta -\lambda e^{-v_0^\delta} F'(u_0+\xi_n) 
	\right) \left( w-\tilde f_n \right)\\
	\end{split}
	\end{equation}
	where $\kappa_n$ and $\xi_n$ are functions on $\Sigma$ gives $\tilde f_1>\tilde f_2>\cdots >\tilde f_n>\cdots >w$ on $\Sigma$. 
\end{proof}

This proposition implies $|\!|\tilde f_{n-1}|\!|_{L^\infty(\Sigma)}$ are uniformly bounded (when $\delta\in (0,1)$ is fixed) and thus by Equation \eqref{eq:iterative} the $C^{1,\gamma}$ norm of $\tilde f_n$ are uniformly bounded for any $\gamma\in (0,1)$. The RHS of the equation 
has a uniform $C^{\gamma}$ bound and therefore the $C^{2,\gamma}$ norm of $\tilde f_n$ are uniformly bounded for any $\gamma\in (0, 1)$.  As a consequence, the sequence uniformly converges as $n\to +\infty$ to a $C^{2,\gamma}$ function $\tilde f^\delta$ solving the equation \eqref{eq:combined} on $\Sigma$. Notice that $\tilde f^\delta$ a priori depends on $\Psi_\sigma, w$ and $\lambda$. 

Looking back to the equation 
\begin{equation}
\label{eq:combined-re}
\Delta \tilde f^\delta 
+ 
\frac{1}{2}\lambda e^{-v_0^\delta } F( 2\tilde f^\delta + u_0^\delta ) 
=
- \frac{2\pi \tilde N}{\text{Vol}_{g_0}}
\end{equation}
and the above construction shows that 
\begin{equation}
w
\leq 
\tilde f^\delta
<
 \tilde f_1
 =- \frac{1}{2}u_0^\delta+\frac{1}{2}\log \tau\;\;,\;\; \forall \delta\in (0,1),
\end{equation}
which in particular implies that for any $q>1$, 
\begin{equation}
|\!|\tilde f^\delta|\!|_{L^q(\Sigma)}
\leq 
C
\end{equation}
for some $C$ independent of $\delta\in (0,1)$. 
Denote   ({\bfseries{A}}) for the analytical assumption:
\begin{equation}
\exists\; p>1, \text{ s.t. } 
e^{-v_0}\in L^p(\Sigma). 
\end{equation}
This assumption implies 
\begin{equation}
\label{ass:Lp-bounded}
\text{$\exists\; C>0$, $p>1$, s.t. } 
|\!| e^{-v_0^\delta} |\!|_{L^p(\Sigma)}
\leq 
C\;\;, \;\;\forall \delta\in (0,1).
\end{equation}
Since $F$ is uniformly bounded on $\mathbb{R}$, \eqref{ass:Lp-bounded} implies there is a uniform constant $C$ independent of $\delta\in (0,1)$ such that 
\begin{equation}
|\!|\tilde f^\delta|\!|_{W^{2,p}}\leq C
\end{equation}
by Equation \eqref{eq:combined-re}. The Sobolev embedding then implies that 
\begin{equation}
|\!| \tilde f^\delta |\!|_{C^{\gamma}}
\leq C
\end{equation}
for a $\delta$-independent constant $C$ and $\gamma\in (0,1)$. The Equation \eqref{eq:combined-re} is then rewritten as 
\begin{equation}
\Delta \tilde f^\delta
=
- \frac{\lambda}{2} \frac{1}{\prod_j (|\bm s_j|^2+\delta)^{1-\beta_j}} 
e^{4\alpha\tau \tilde f^\delta - 2\alpha e^{2\tilde f^\delta+ u_0^\delta}} (e^{2\tilde f^\delta + u_0^\delta}-\tau)
- 
\frac{2\pi \tilde N}{\text{Vol}_{g_0}}
\end{equation}
whose RHS is easily seen to be uniformly bounded in $L^{p'}$ for some $p'>1$. The consequence is that 
\begin{equation}
|\!|\tilde f^\delta|\!|_{C^{1,\gamma'}(\Sigma)}
\leq 
C
\end{equation}
for some $C$ independent of $\delta\in (0,1)$.  On any compact subset $K$ away from $\mathcal{S}$ we have $\Delta \tilde f^\delta$ is uniformly bounded in $C^{\gamma'}(K)$ and therefore $\tilde f^\delta$ is bounded in $C^{2,\gamma'}(K)$. By taking  an exhaustion of $\Sigma\backslash \mathcal{S}$ and taking a diagonal subsequence as $\delta\to 0$, we have a limit function $\tilde f$ in $C^2_{loc}(\Sigma\backslash \mathcal{S})\cap C^{1,\gamma'}(\Sigma)$ solving the Equation \eqref{eq:singular-Bogomolnyi}. 

We divide the singularity set $\mathcal{S}$ into seven groups according to the Venn diagram of the three sets $A=\mathcal{Z}$, $B=\mathcal{C}$ and $C=\mathcal{P}$.  The analytical assumption ({\bf{A}}) is satisfied if and only if the the numerical assumption ({\bf{N}}) about any singular point $x$ is satisfied for $x$ in the corresponding part of the Venn diagram  (Figure 1), for instance if $x\in (\mathcal{Z}\cap \mathcal{C})\backslash \mathcal{P}$ saying $x=p_k=q_k$ then we should have $4\alpha\tau n_k +2(1-\beta_k) <2$.  We expect there is a stability explanation to these numerical assumptions as what happens in the smooth case \cite{Al-Ga-Ga-P, FPY}. 

\begin{theorem}[Existence of singular cosmic strings]
	\label{thm:singular-EB}
	Let $L$ be a holomorphic line bundle on $\mathbb{P}^1$, suppose $c_1(L)=N$ and $\bm\phi$ is a holomorphic section of $L$ with the zero set $\mathcal{Z}$. Let $\mathcal{C}, \mathcal{P}$ be two subsets with degrees, such that the numerical assumption ({\bf{N}}) is satisfied, then there exists singular gravitating vortex solution to the Equation \eqref{eq:cpGV} when the coupling constant and symmetry breaking parameter satisfies $\alpha\tau = \frac{\tilde \chi}{2\tilde N}= \frac{2-\sum_j (1-\beta_j)}{2(N+\sum_k \alpha_k)}$.  
\end{theorem}


\begin{figure}
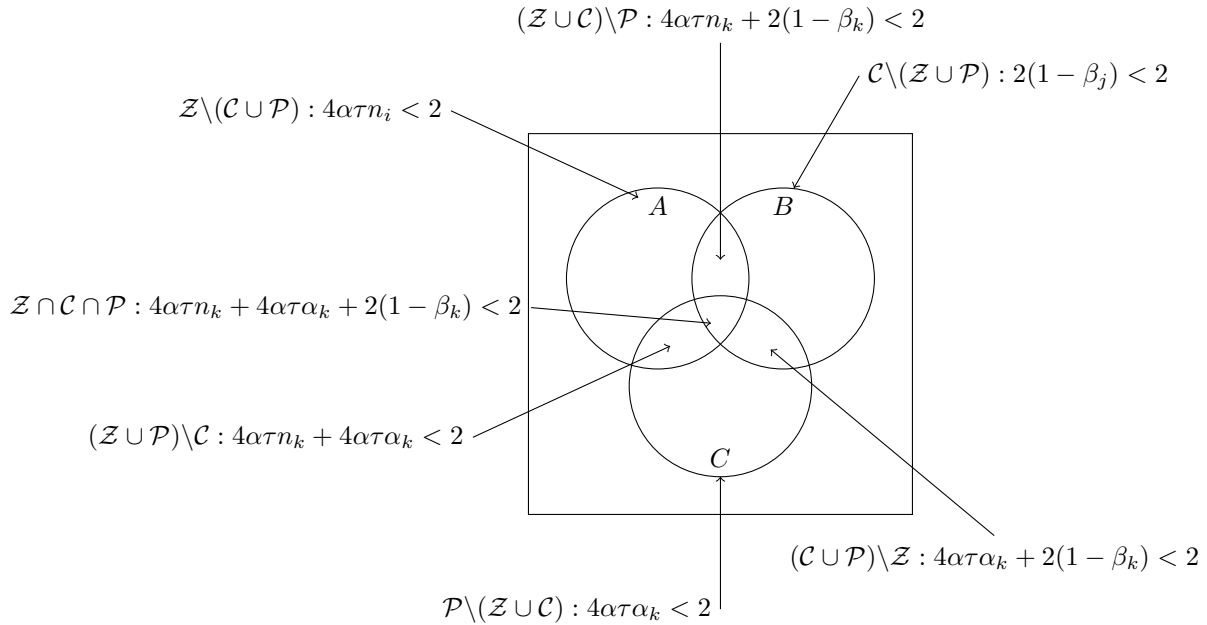

	\label{fig:venn-diagram}
\begin{venndiagram3sets}[labelOnlyA={},labelOnlyB={},labelOnlyC={}, labelOnlyAB={},labelOnlyAC={},labelOnlyBC={},labelABC={}]
	\setpostvennhook
	{
		\draw[<-] (labelA) -- ++(155:3cm) node[left] {$\mathcal{Z}\backslash (\mathcal{C}\cup \mathcal{P}): 4\alpha\tau n_i <2$};
		\draw[<-] (labelB) -- ++(60:2cm) node[right] {$\mathcal{C}\backslash (\mathcal{Z}\cup \mathcal{P}): 2(1-\beta_j)<2$};
		\draw[<-] (labelC) -- ++(-90:2cm) node[left] {$\mathcal{P}\backslash (\mathcal{Z}\cup \mathcal{C}): 4\alpha\tau \alpha_k <2$};
				\draw[<-] (labelOnlyAB) -- ++(90:3cm) node[above] {$(\mathcal{Z}\cup \mathcal{C})\backslash \mathcal{P} : 4\alpha\tau n_k + 2(1-\beta_k)<2$};
						\draw[<-] (labelOnlyBC) -- ++(-40:4cm) node[below] {$(\mathcal{C} \cup \mathcal{P})\backslash \mathcal{Z}: 4\alpha\tau \alpha_k +2(1-\beta_k)<2$};
								\draw[<-] (labelOnlyAC) -- ++(-155:3cm) node[left] {$(\mathcal{Z}\cup \mathcal{P})\backslash \mathcal{C}: 4\alpha\tau n_k + 4\alpha\tau \alpha_k <2$};
		\draw[<-] (labelABC) -- ++(175:2.5cm) node[left]
		{$\mathcal{Z}\cap \mathcal{C}\cap \mathcal{P}: 4\alpha\tau n_k + 4\alpha\tau \alpha_k + 2(1-\beta_k)<2$}; 
	}
\end{venndiagram3sets}
	\caption{Numerical assumption for various group of singularities}
\end{figure}

\begin{remark}
	Even though the solution $g$ to singular Einstein-Bogomol'nyi equations is conformal to the starting smooth metric $g_0$, they are not necessarily with the same volume. The metric resulted from the above construction seems to depend on $\lambda$. Actually, the uniqueness of solution is a subtle issue as indicated by \cite{HS}, and is naturally a quite interesting question thinking of the infinite dimensional moment map picture \cite{Al-Ga-Ga1, Al-Ga-Ga2} and the physical background regarding cosmic strings.   
	\end{remark}

\end{document}